\documentclass[A4]{amsart}
\usepackage[square,sort,comma,numbers]{natbib}
\usepackage{amsmath, amssymb, amstext, amsfonts, textcomp, amsxtra, amsbsy, amsgen, amsopn, amscd, mathrsfs, amsthm, latexsym, array}
\usepackage[textwidth=16cm,textheight=22cm,centering]{geometry}

\usepackage{enumerate}

\usepackage{bbm, dsfont}
\usepackage{color}
\usepackage{graphicx}

\usepackage{xcolor}
\definecolor{cite}{rgb}{0.00,0.60,1.00}
\definecolor{url}{rgb}{1.00,0.10,0.80}
\definecolor{link}{rgb}{0.00,0.00,1.00}
\usepackage[colorlinks,linkcolor=link,urlcolor=url,citecolor=cite,pagebackref,breaklinks]{hyperref}

\def\leq{\leqslant}

\def\geq{\geqslant}

\allowdisplaybreaks

\usepackage[all]{xy}

\newtheorem{theorem}             {Theorem}  [section]
\newtheorem{definition} [theorem] {Definition}
\newtheorem{lemma}      [theorem]{Lemma}
\newtheorem{corollary}  [theorem]{Corollary}
\newtheorem{proposition}[theorem]{Proposition}

\numberwithin{equation}{section} 

\theoremstyle{remark}
\newtheorem{remark}{\bf Remark}

\newcommand{\Cont}{\mathcal{C}}

\newcommand{\Sch}{\mathcal{S}}
\newcommand{\sgn}{{\rm sgn}}
\newcommand{\intL}{{\rm L}}

\newcommand{\Nr}{{\rm Nr}}
\newcommand{\Tr}{{\rm Tr}}

\newcommand{\FBessel}[3]{\mathfrak{j}(#1; #2, #3)}

\newcommand{\gp}[1]{\mathbf{#1}}
\newcommand{\GL}{{\rm GL}}
\newcommand{\PGL}{{\rm PGL}}
\newcommand{\SL}{{\rm SL}}
\newcommand{\SO}{{\rm SO}}
\newcommand{\Sp}{\mathrm{Sp}}
\newcommand{\SU}{{\rm SU}}

\newcommand{\id}{\mathbbm{1}}

\newcommand{\ag}[1]{\mathbb{#1}}

\newcommand{\Z}{\mathbb{Z}}

\newcommand{\Mat}{{\rm M}}

\newcommand{\Q}{\mathbb{Q}}
\newcommand{\R}{\mathbb{R}}
\newcommand{\C}{\mathbb{C}}
\newcommand{\E}{\mathbb{E}}
\newcommand{\F}{\mathbb{F}}

\newcommand{\A}{\mathbb{A}}

\newcommand{\vo}{\mathfrak{o}}

\newcommand{\vp}{\mathfrak{p}}

\newcommand{\idlN}{\mathfrak{N}}

\newcommand{\Dis}{{\rm D}}
\newcommand{\Dif}{\mathfrak{D}}

\newcommand{\norm}[1][\cdot]{\lvert #1 \rvert}
\newcommand{\extnorm}[1]{\left\lvert #1 \right\rvert}

\newcommand{\Pairing}[2]{\langle #1, #2 \rangle}
\newcommand{\extPairing}[2]{\left\langle #1, #2 \right\rangle}

\newcommand{\OFour}{\mathfrak{F}}

\newcommand{\rpR}{{\rm R}}
\newcommand{\Bas}{\mathcal{B}}
\newcommand{\Hom}{{\rm Hom}}
\newcommand{\Ind}{{\rm Ind}}
\newcommand{\cInd}{{\rm ind}}

\newcommand{\Intw}{\mathcal{M}}

\newcommand{\Whi}{\mathcal{W}}
\newcommand{\Kir}{\mathcal{K}}
\newcommand{\Cond}{\mathbf{C}}
\newcommand{\cond}{\mathfrak{a}}

\newcommand{\fin}{{\rm fin}}
\newcommand{\eis}{{\rm E}}
\newcommand{\Reis}{\mathcal{E}}

\newcommand{\reg}{{\rm reg}}

\newcommand{\Zeta}{\mathrm{Z}}
\newcommand{\Mt}{\mathrm{M}}
\newcommand{\Wt}{\mathrm{Wt}}
\newcommand{\Dt}{\mathrm{D}}

\newcommand{\Vol}{{\rm Vol}}
\makeatletter

\newcommand{\Rmnum}[1]{\expandafter\@slowromancap\romannumeral #1@}
\makeatother

\newcommand{\ud}{\mathrm{d}}
\newcommand{\ue}{\mathrm{e}}

\newcommand{\tr}{\mathrm{tr}}

\newcommand{\rank}{\mathrm{rank}}
\newcommand{\Swan}{\mathrm{Swan}}

\newcommand{\bx}{\mathbf{x}}
\newcommand{\by}{\mathbf{y}}

\newcommand{\bE}{\mathbf{E}}
\newcommand{\bF}{\mathbf{F}}

\newcommand{\bR}{\mathbf{R}}

\newcommand{\cC}{\mathcal{C}}

\newcommand{\cF}{\mathcal{F}}

\newcommand{\cH}{\mathcal{H}}

\newcommand{\sL}{\mathscr{L}}

\newcommand{\fa}{\mathfrak{a}}

\newcommand{\fc}{\mathfrak{c}}

\newcommand{\fn}{\mathfrak{n}}

\newcommand{\ft}{\mathfrak{t}}

\newcommand{\fG}{\mathfrak{G}}

\DeclareMathOperator*{\ord}{ord}
\DeclareMathOperator*{\Res}{Res}

\DeclareMathOperator{\bFrob}{Frob}

\DeclareFontFamily{U}{mathx}{\hyphenchar\font45}
\DeclareFontShape{U}{mathx}{m}{n}{
      <5> <6> <7> <8> <9> <10>
      <10.95> <12> <14.4> <17.28> <20.74> <24.88>
      mathx10
      }{}
\DeclareSymbolFont{mathx}{U}{mathx}{m}{n}
\DeclareMathAccent{\widecheck}{\mathalpha}{mathx}{"71}

\def\leq{\leqslant}

\def\geq{\geqslant}

\title{A uniform Weyl bound for $L$-functions of Hilbert modular forms}

\author{Han Wu}
\address{School of Mathematical Sciences, University of Scinece and Technology of China, 230026 Hefei, P. R. China}
\email{wuhan1121@yahoo.com}

\author{Ping Xi}
\address{School of Mathematics and Statistics, Xi'an Jiaotong University, Xi'an 710049, P. R. China}
\email{ping.xi@xjtu.edu.cn}

\date{\today}

\begin{document}



\begin{abstract}
We establish a Weyl-type subconvexity of $L(\tfrac{1}{2},f)$ for spherical Hilbert newforms $f$ with level ideal $\idlN^2$, in which $\idlN$ is required to be cube-free, and at any prime ideal $\vp$ with $\vp^2 \mid \idlN$ the local representation generated by $f$ is not supercuspidal. The proof exploits a distributional version of Motohashi's formula over number fields developed by the first author, as well as Katz's work on hypergeometric sums over finite fields in the language of $\ell$-adic cohomology.
\end{abstract}

	\maketitle

\tableofcontents

\section{Introduction and Main Results} \label{sec:introduction-results}

\subsection{Background}
Given a number field $\bF$ with adelic ring $\A$, it is fundamental to understand the growth of the $L$-function $L(\tfrac{1}{2},\pi)$ as the automorphic representation $\pi$ of $\GL_n(\A)$ varies in a suitable family. The Phragm\'en--Lindel\"of principle gives the convex bound $L(\tfrac{1}{2},\pi)\ll\Cond(\pi)^{1/4+\epsilon}$ for any $\epsilon>0,$ where $\Cond(\pi)\in\bR_{\geqslant1}$ denotes the analytic conductor of $\pi,$ and the implied constant depends on $\epsilon,n$ and $\bF$. The Lindel\"of hypothesis for $L(s,\pi)$ asserts that $1/4$ in the above exponent can be removed. Motivated by many applications, we are interested in reducing the exponent $1/4$, i.e., proving a subconvex bound for $L(\tfrac{1}{2},\pi)$. The quality of subconvex bounds for $L$-functions measures the distance between the current mathematical technology and the Lindel\"of hypothesis. The first subconvexity is due to Weyl \cite{Wyl21}, Hardy--Littlewood \cite{Li22} and Landau \cite{La24} that the Riemann zeta function satisfies
\begin{equation} \label{RZWeylBd}
	\zeta(\tfrac{1}{2}+it) \ll_{\epsilon} (1+\norm[t])^{1/6+\epsilon}.
\end{equation}
For other families of $L$-functions, there are quite limited instances with satisfactory subconvex bounds established. On the other hand, in some favorable cases of $\pi$, one is even able to prove the Weyl-type bound
\begin{equation} \label{ALWeylBd}
	L(\tfrac{1}{2}, \pi) \ll_{\epsilon, n,\bF} \Cond(\pi)^{1/6+\epsilon}.
\end{equation}
Note that the analytic conductor is decomposed as a product indexed by the places $v$ of $\bF$:
\begin{align*}
\Cond(\pi)=\prod_v \Cond(\pi_v)
\end{align*}
for $\pi = \otimes_v' \pi_v$.
We call \eqref{ALWeylBd} a \emph{uniform Weyl bound} to emphasize its uniformity in all places $v$\footnote{We do not include the uniformity with respect to the discriminant of $\bF$, the dependence on which is allowed to be polynomial.}. The exponent $1/6$ turns out to be a natural barrier very difficult to reach or break. Besides the above-mentioned bound for Riemann zeta functions, the so far published uniform Weyl bound are:
\begin{itemize} 
	\item for Dirichlet $L$-functions $L(\tfrac{1}{2}+it,\chi)$: Heath-Brown \cite{HB78} requires the conductor of $\chi$ to have very nice factorization, and the general situation was successfully settled in the recent work by Petrow and Young \cite{PY19_CF, PY19_All}, non-trivially extending the method of Conrey and Iwaniec \cite{CI00} and generalizing the work of Young \cite{Yo17} for quadratic character $\chi$;
	\item for Dedekind $L$-functions $L(\tfrac{1}{2}+it,\chi)$: Soehne \cite{Soe97}, following the method of Heath-Brown, also requires the conductor of $\chi$ to have very nice factorization.
\end{itemize}	
	There are also some other non-uniform Weyl bounds, including
\begin{itemize}
	\item Good \cite{Go82} and Meurman \cite{Me90} for $\pi = \pi_0 \otimes \norm_{\A}^{it}$ with $\pi_0$ corresponding to a fixed holomorphic/Maass form for $\SL_2(\Z)$ in the $t$-aspect, and Jutila--Motohashi \citep{JM05} for varying $\pi_0$ in the hybrid aspect;
	\item Lau--Liu--Ye \cite{LLY06} for twisted Rankin-Selberg products $\pi = (\pi_1 \times \pi_2) \otimes \norm_{\A}^{it}$ with fixed $\pi_1$ in the archimedean aspect for $\pi_2$;
	\item Blomer--Jana--Nelson \cite{BJN22+} for triple products $\pi = \pi_1 \times \pi_2 \times \pi_3$ with $\pi_j$ corresponding to modular forms for $\SL_2(\Z)$, $\pi_1$ and $\pi_2$ fixed and in the archimedean aspect for $\pi_3$.
\end{itemize}

\subsection{Main result}
	In this paper, we establish a new instance of uniform Weyl bounds, which can be formulated in the adelic language as follows.
\begin{theorem} \label{Main}
	Let $\bF$ be a totally real number field with adelic ring $\A$. For $\pi$ an automorphic representation of $\PGL_2(\A)$ such that 
\begin{itemize}
	\item at every real place $v$ the local component $\pi_v$ is spherical with respect to $\SO_2(\R),$
	\item at every finite place $\vp$ the conductor exponent satisfies $\cond(\pi_{\vp}) \in \{ 2, 4 \}$ and is not supercuspidal if $\cond(\pi_{\vp}) = 4$,
\end{itemize}	
	the uniform Weyl bound $\eqref{ALWeylBd}$ holds for any $\epsilon > 0.$
\end{theorem}

\begin{remark}
	The conductor exponent is defined by $\Cond(\pi_{\vp}) = \Nr(\vp)^{\cond(\pi_{\vp})}$. Moreover, if $\pi$ is generated by the unique Hilbert newform, then it has the level ideal 
	$$ \idlN = \prod_{\vp < \infty} \vp^{\cond(\pi_{\vp})}. $$
	Hence Theorem \ref{Main} is translated into the classical language in terms of Hilbert modular forms.
\end{remark}

\begin{remark} \label{rmk:PYFamily}
	Our result is completely new even over $\Q$, since we are able to deal with supercuspidal representations. Petrow and Young \cite{PY19_All} proved a uniform Weyl bound for some family of $\PGL_2$ automorphic representations $\pi$ over $\Q$, in which for every local component $\pi_p$ one can find a character $\xi_p$ of $\Q_p^{\times}$ with conductor exponent $n_p$, such that $\cond(\pi_p \otimes \xi_p) \leq n_p$. But for a supercuspidal $\pi_p$ of $\PGL_2(\Q_p)$ with $\cond(\pi_p) = 2$, we have
	$$ \cond(\pi_p \otimes \xi_p) = \begin{cases} \cond(\pi_p)=2 ,& \text{if } n_p \leq 1, \\ 2n_p, & \text{if } n_p \geq 2, \end{cases}
	$$
	hence $\pi_p$ does not satisfy the local condition as required by Petrow--Young's family. The same is true for the family treated in a previous work of Balkanova, Frolenkov and Wu \cite{BFW21+}.
\end{remark}

	As in many existing arithmetic applications, subconvex bounds for $L$-functions are related to equidistributions of geometric/arithmetic objects in families, and sharper exponents automatically imply better rate of convergence. See, for instance,
\begin{itemize}	
	\item the Park City lecture notes of Michel \cite[Lecture 5]{Mi07} including equidistributions of Heegner points and the Quantum Chaos problems,
	\item the error term bounds for the number of integral representations of integral ternary quadratic forms \cite[Corollary 2]{BH10},
	\item the error term bounds in prime geodesic theorems \cite{SY13, BF18, BBCL20}.
\end{itemize}
More fascinating phenomena consist of the essential requirement of either \emph{uniformity} or/and \emph{Weyl-type} quality of subconvexity in certain applications: See the work of Andersen and the first author \cite{AW22} on the partition function for ``uniformity'', and the work of Ghosh and Sarnak \cite{GS11} and Matom\"aki \cite{Ma12} on real zeros of holomorphic cusp forms for ``Weyl-type''.

	To prove (uniform) Weyl bounds, Motohashi's formula on moments of $L$-functions turns out to be a very crucial and 
fundamental tool. From 1990's, Motohashi \cite{Mo93, Mo97} developed a summation formula relating the fourth moment of the Riemann zeta functions and the cubic moment of the modular $L$-functions for the full modular group $\SL_2(\Z)$. In the inverse direction, formulae of such type have been utilized to study cubic moments and thus Weyl-type bounds for modular $L$-functions by Conrey and Iwaniec \cite{CI00}, Young \cite{Yo17}, Petrow \cite{Pe15} and Petrow and Young \cite{PY18, PY19_CF, PY19_All}.

	Recently, a general version of Motohashi's formula over number fields has been established by the first author \cite{Wu22}, which was subsequently applied by Balkanova, Frolenkov and Wu \cite{BFW21+} to generalize Petrow--Young's Weyl bound \cite{PY19_CF} to cube-free Dirichlet characters over totally real number fields. 
Note that all the above-mentioned Weyl bounds are proven for $\GL_1$ $L$-functions with Dirichlet or Hecke characters, and it is desirable to see for which family of $\PGL_2$ $L$-functions Motohashi's formula can establish Weyl type bounds. Theorem \ref{Main} adds one supercuspidal representation to the family, extending our former result \cite{BFW21+}. We expect that Motohashi's formula could be utilized to establish Weyl bounds for \emph{all} $\PGL_2$ $L$-functions, but this seems too ambitious given the current version of Motohashi's formula.
In fact, the case of $\PGL_2$ cuspidal representations with prime conductor cannot be covered in the work of 
Petrow and Young \cite{PY19_CF} merely using Motohashi's formula, and it is Blomer, Humphreis, Khan and Milinovich \cite[Corollary 3]{BHKM20} who are able to explore subconvexity for $\PGL_2$ $L$-functions for cuspidal representations with prime conductor by introducing an extra amplification process.

It is more practical to expect that a Weyl-type bound can be proven for a cuspidal $\pi$ of $\PGL_2$ via a (Lindel\"of consistent estimate) of cubic moment over a family $\mathcal{F}$ if the size of the family satisfies $|\cF|\asymp \Cond(\pi)^{1/2}$. If we restrict to sub-families of $\pi$ with spherical archimedean components and conductor exponents $\cond(\pi_{\vp}) \leq 4$, our computation seems to identify all those componentwise-defined $\cF$ with the property $|\cF|\asymp \Cond(\pi)^{1/2}$ for each $\pi \in \cF$. Precisely, we establish the following result which strengthens Theorem \ref{Main}, and shows the limit of the method of cubic moments via Motohashi's formula in \cite{Wu22}.
	
\begin{theorem} \label{MainBis}
	Let $\bF$ be a totally real number field with adelic ring $\A$. For $\pi$ an automorphic representation of $\PGL_2(\A)$ such that 
\begin{itemize}
	\item at every real place $v$ the local component $\pi_v$ is spherical with respect to $\SO_2(\R),$
	\item at every finite place $\vp$ the conductor exponent satisfies $\cond(\pi_{\vp}) \leq 4,$ and $\pi_{\vp}$ is not supercuspidal if $\cond(\pi_{\vp}) = 4$.
\end{itemize}	
	Let $S_j=S_j(\pi)$ be the set of finite places $\vp$ such that $\cond(\pi_{\vp})=j$, and write 
	$$ \Cond_j(\pi) = \prod_{\vp \in S_j} \Cond(\pi_{\vp}). $$
	Then we have the following bound
	$$ L( \tfrac{1}{2}, \pi) \ll_{\epsilon, \bF} \Cond_1(\pi)^{1/6} \Cond_3(\pi)^{1/18} \Cond(\pi)^{1/6+\epsilon}. $$		
\end{theorem}

\begin{remark}
As one may see from the proof of Theorem \ref{MainBis}, an interesting phenomenon is that the main contribution on the dual side of the cubic moment formula does not always come from the fourth moment but from some degenerate terms (residues of the dual weights at special points), and manifests no uniform quality. This is in contrast with Michel--Venkatesh's work \cite{MV10}, where the total contribution from all the degenerate terms cannot surpass those from the dual moment. Note also that the dual moment itself has a good size for the Weyl bound.
\end{remark}

\subsection{Structure of this paper}
	
	The proof of Theorem \ref{Main} is based on applications of Motohashi's formula developed by the first author \cite{Wu22}. We will recall the precise form of this formula in Section \ref{sec:Motohashiformula} (Theorem \ref{thm:Motohashi}).
	
	For the convenience of readers not familiar with the language of representation theory, we recollect the relevant notation in Section \ref{sec:representationtheory} while giving a brief survey on the relevant theories of zeta integrals. Some convenient references include Lang \cite[Chapter \Rmnum{14}]{Lan03} for Tate's thesis, Goldfeld and Hundley \cite{GoJ11} for the Godement--Jacquet theory, and Gelbart \cite{Ge75} for the Rankin--Selberg theory for $\GL_2 \times \GL_1$ (due to Hecke and Jacquet--Langlands). The idea of invariant distributions contained in Weil's re-interpretation of Tate's theis is emphasized along our presentation. Weil's idea guided the discovery of the first author's version of Motohashi's formula in \cite{Wu22}. It should also be helpful for readers to understand it with some depth.

	The general strategy of the proof amounts to modifying the test functions locally in order to select the desired families of $L$-functions on the cubic moment side. The first main innovation in this paper is the specification of test functions. Compared to Petrow--Young's approach, our new version of Motohashi's formula with local-global features allows us to work with all test functions locally as in \cite{AW22}. In Section \ref{sec:localcomputations}, we classify those local representations $\pi_{\vp}$ relevant to Theorem \ref{Main}, choose the test function in each case, and present all necessary bounds of the local dual weight functions. The case of depth-zero supercuspidal representations is the most difficult part and illustrates another very interesting feature in this paper, for which we may reduce the problem to estimating the double character sum
\begin{align*}
\sum_{\alpha\in\F_q} \rho(\alpha+\omega)\sum_{t\in\F_q} \chi(t)\phi(\alpha^2-\omega^2 t)\phi(1-t)
\end{align*}
defined over the finite field $\F_q$. Here $\rho$ (resp. $\chi$) is a non-trivial character of $\F_{q^2}^{\times}$ (resp. $\F_q^{\times}$), and $\phi$ is the quadratic character of $\F_q^{\times}$.
This sum has its origin in arithmetic geometry, and will be treated using Katz's work on hypergeometric sums, for which Deligne's proof on Weil's conjectures plays an essential role. These details will be given in Section \ref{sec:charactersums} with the necessary theory developed by Katz \cite{Ka90}. Since this part has a different feature, it can be read independently, and different notation will be used. To understand why the above double character sum appears in our current work, it should be better to mention that another double character sum
\begin{align*}
\mathop{\sum\sum}_{u,v\in\F_q}\chi(u)\overline{\chi}(u+1)\overline{\chi}(v)\chi(v+1)\eta(uv-1)
\end{align*}
appears naturally in the work of Conrey--Iwaniec \cite{CI00} and Petrow--Young \cite{PY19_CF, PY19_All}, which has been shown in \cite{Xi23} to be essentially a hypergeometric sum, and only an upper bound with squareroot cancellations is sufficient therein! In Section \ref{sec:proofcompleted}, we recollect the local bounds to prove the main result Theorem \ref{Main} via our Motohashi's formula.

Motohashi's formula received considerable attentions in recent years guided by applications to subconvexity problems for $L$-functions. A period approach to understanding this formula has been developed by Michel and Venkatesh \cite{MV10} and Nelson \cite{Ne20}. We initiated a comparison between the distributional version and Nelson's in \cite[\S 7 Appendix]{Wu19_TF}. We shall continue and refine the comparison in Appendix \ref{sec:appendix}, clarifying a non-trivial gap from Nelson's version to the distributional one. 

\begin{remark} \label{rmk:Hu-Petrow-Young}
	For supercuspidal representations with larger conductor exponents, Hu, Petrow and Young have been investigating some cases over $\Q$. We also have some local computations in the case $\cond(\pi_{\vp}) \geq 4$. It seems that a tight bound for the dual weights does not suffice for the corresponding Weyl bound, unlike the cases treated by Petrow and Young \cite{PY19_CF, PY19_All}. In particular, this is why we have excluded supercuspidals with conductor exponent $4$. There is still a highly non-trivial difficulty in proving Weyl type bounds for all \emph{suitable} $\PGL_2$ $L$-functions, whose resolution would reach beyond the spectral reciprocity offered by the current version of Motohashi's formula.
\end{remark}

	\subsection{Conversations with weight functions in spectral reciprocity}
	Before closing this sections, we would like to leave some comments and remarks on choosing and bounding the weight functions in applications of spectral reciprocity formulas, like Motohashi's formula. The readers can ignore this part for now, and come back to such issues after all the subsequent treatments have been checked.
	
	We now mention two different opinions in treating the weight functions on both sides. 
\begin{itemize}
	\item  One opinion, focused on applications to the subconvexity problem,  looks for test functions on the relevant groups/algebras, and bounds the weight functions on each side qualitatively. A typical example is the celebrated work of Michel and Venkatesh \cite{MV10} on the subconvexity for $\GL_2$, in which the asymptotic behaviors of the weight functions can be conveniently explained in terms of certain equidistribution properties.

	\item The other one looks for specifying the admissible weight function on one side of the reciprocity formula, and bound the dual weight function via establishing an explicit transformation formula. This follows Motohashi's original paper \cite{Mo93}, and has advantage towards the moment problems.
\end{itemize}

	Our method in this paper is a mixture of the above two opinions: at real places we follow the second opinion while at non-archimedean places the first one helps. Some further discussions on this mixture are interesting and necessary.
	
	As mentioned above, a very interesting feature contained in our treatment is the application of Katz's theory of hypergeometric sums to bounding the dual weight functions. Hence Deligne's proof of Weil's conjectures is applied in an essential way. In this lower rank group case of $\GL_2$, the computation leading to the relevant exponential sums seems to be easier with the first opinion. However, some phenomenon looks mysterious in this way: the treatment of the local dual weight function in the simple supercuspidal case, or equivalently the case of local conductor exponent $\cond(\pi_{\vp})=3$, is much simpler than the treatment in the depth $0$ case, and even simpler than the treatment in the case of Petrow--Young's family mentioned in Remark \ref{rmk:PYFamily}. This phenomenon seems to be better understood with the local weight transformation process conjectured in \cite[(1.14)]{BFW21+}. In fact, our test function selects only the relevant supercuspidal representation $\pi_{\vp}$, hence this three-step process simplifies to a two-step one:
\begin{equation} \label{eq:GenLocInv}
	\begin{cases}
	H(t) =\displaystyle\int_{\widehat{\bF^{\times}}} \xi(t)\norm[t]_{\vp}^c \gamma(c,\xi_{\vp},\psi_{\vp})^3 \ud \mu_{\mathrm{PL}}(\xi) \int_{\bF^{\times}} \FBessel{y}{\pi_{\vp}}{\psi_{\vp}} \xi(y) \norm[y]_{\vp}^c \ud^{\times}y, \\
	\widetilde{h}(\chi_{\vp}) = \chi_{\vp}(-1)\displaystyle \int_{\bF^{\times}} H(t) \psi_{\vp}(t) \chi_{\vp}^{-1}(t) \norm[t]_{\vp}^{-1/2} \ud^{\times}t,
	 \end{cases}
\end{equation}
	where $\bF = \bF_{\vp}$ is the completion of $\bF$ at $\vp$, and $\FBessel{y}{\pi_{\vp}}{\psi_{\vp}}$ is the Bessel function of $\pi_{\vp}$. Now that the first integral
	$$ \int_{\bF^{\times}} \FBessel{y}{\pi_{\vp}}{\psi_{\vp}} \xi(y) \norm[y]_{\vp}^c \ud^{\times}y = \gamma(\tfrac{1}{2}-c,\pi_{\vp} \times \xi_{\vp}^{-1}, \psi_{\vp}) $$
	is simply a twisted local gamma factor of $\pi_{\vp}$, the complexity of this transformation is measured by the complexity of these twisted local gamma factors. It is also not hard to believe that the complexity is only concerned with characters $\xi_{\vp}$ with conductor exponent $1$, since the most difficult parts are always concerned only with those exponential sums defined over the residual class field $k_{\bF}$.
\begin{itemize}
	\item For simple supercuspidals, the relevant twisted local gamma factors are computed in \cite[Corollary 3.12]{AL16}. The dependence on the twisting character $\xi$ is very mild, essentially only $\xi_{\vp}(\varpi)$ is involved.
	\item For principal series representations in Petrow--Young's family, the relevant twisted local epsilon factors are products of two Gauss sums defined over $k_{\bF}$, which is a standard fact. While the local gamma factors would include an extra $\zeta_{\vp}(s)$ for two special $\xi_{\vp}$.
	\item For depth $0$ supercuspidal $\pi_{\vp}$, the relevant twisted local gamma factors are computed in \cite[(25.3.1) \& (25.4.1)]{BuH06}. Gauss sums over $k_{\bE}$, the unramified quadratic extension over $k_{\bF}$, are involved.
\end{itemize}
	We do observe the same order of complexities betwen the twisted local gamma factors and dual weight functions. 
	
	We also note that, if we write the general local weight transformation process in a single formula as
	
\begin{equation} \label{eq:LocInvFGen}
	\widetilde{h}(\chi_{\vp}) = \int_{\widehat{\PGL_2(\bF)}} h(\sigma) K(\sigma, \chi_{\vp}) d\mu_{\mathrm{PL}}(\sigma),
\end{equation}

\noindent and if the admissible weight function $h(\sigma)$ is non-zero at only one discrete series representation $\pi_{\vp}$, then bounding $\widetilde{h}(\chi_{\vp})$ is precisely the same as bounding the kernel function $K(\pi,\chi_{\vp})$. It would be very interesting to see how other ideas such as (some non-archimedean analogue of) microlocalized vectors, which always select $\pi$ and other $\sigma$ in the principal series with comparable conductor with $\pi$, can help bypass Katz's theory or help lead to it more quickly.
	
%
	

\subsection*{Acknowledgements}
We thank Bingrong Huang for pointing out a mistake in an earlier version, and we are also grateful to Yang Cao for his helpful discussions on Lang torsors.
This project was initiated when HW was a postdoc at QMUL in early 2021, resumed when HW visited XJTU during November 2022, and was completed when PX visited USTC during February 2023. We thank these institutions for hospitality.
The work is supported in part by NSFC (No. 12025106, No. 11971370). HW was also supported by the Leverhulme Trust Research Project Grant RPG-2018-401.

\section{Integral Representations of $L$-Functions from Representation Theory}
\label{sec:representationtheory}
	
	\subsection{Recall on Tate's thesis}
	
	For a locally compact abelian group $G$, we write $\widehat{G}$ for the dual group of continuous unitary characters.

	Let $\bF$ be a number field with different $\Dif_{\bF}$ and discriminant $\Dis_{\bF} = \Nr_{\bF/\Q}(\Dif_{\bF})$, ring of adeles $\A$, and group of ideles $\A^{\times}$. Write $\A^{(1)}$ for the subgroup of ideles with adelic norm $1$. We identify $\R_{>0}$ with the image of a fixed section map of the adelic norm map $\bF^{\times} \backslash \A^{\times} \to \R_{>0}$, so that $\bF^{\times} \backslash \A^{\times} \simeq \bF^{\times} \backslash \A^{(1)} \times \R_{>0}$ is identified as the direct product of a compact abelian group and $\R_{>0}$. Let $V_{\bF}$ be the set of all places of $\bF$. We fix the non-trivial additive character $\psi: \bF \backslash \A \to \C^1$ \`a la Tate, and choose the Haar measure $\ud x =\prod_v\ud x_v$ on $\A$ to be self-dual with respect to $\psi$. The Haar measure $\ud^{\times} x =\prod_v\ud^{\times}x_v$ on $\A^{\times}$ is taken to be the Tamagawa measure with factors of convergences $\zeta_v(1)$, namely
	\begin{align*}
	\ud^{\times}x_v = \zeta_v(1) \frac{\ud x_v}{\norm[x_v]_v}, \quad 
	\zeta_v(s) := \begin{cases} \pi^{-s/2} \Gamma(s/2),\ \ & \text{if } \bF_v = \R,\\
	(2\pi)^{1-s} \Gamma(s), & \text{if } \bF_v = \C,\\
	(1-\Nr(\vp)^{-s})^{-1}, & \text{if } v=\vp < \infty.\end{cases}
	\end{align*}
	
	Let $\chi = \otimes_v' \chi_v: \R_{>0} \bF^{\times} \backslash \A^{\times} \to \C^1$ be a Hecke character. The goal of Tate's thesis is to establish the meromorphic continuation and functional equation satisfied by the Hecke $L$-function $L(s,\chi)$. This is achieved with auxiliary (decomposable) Bruhat--Schwartz functions $\Phi = \otimes_v' \Phi_v \in \Sch(\A)$ by the fundamental equations
\begin{align}
\Zeta(s, \Phi ,\chi)
&:= \int_{\bF^{\times} \backslash \A^{\times}} \Big(\sum_{\alpha \in \bF^{\times}}\Phi(\alpha t) \Big) \chi(t) \norm[t]_{\A}^s \ud^{\times}t \nonumber \\
&= \int_{\substack{\bF^{\times} \backslash \A^{\times} \\ \norm[t]_{\A} \geq 1}} \Big(\sum_{\alpha \in \bF^{\times}}\Phi(\alpha t) \Big) \chi(t) \norm[t]_{\A}^s \ud^{\times}t + \int_{\substack{\bF^{\times} \backslash \A^{\times} \\ \norm[t]_{\A} \geq 1}} \Big(\sum_{\alpha \in \bF^{\times}}\widehat{\Phi}(\alpha t) \Big) \chi^{-1}(t) \norm[t]_{\A}^{1-s} \ud^{\times}t \label{TateIntG} \\
&\quad + \delta_{\chi = \id} \Vol(\bF^{\times} \backslash \A^{\times}) \Big( \frac{\widehat{\Phi}(0)}{s-1} - \frac{\Phi(0)}{s} \Big) \nonumber \\
&= \int_{\bF^{\times} \backslash \A^{\times}} \Big(\sum_{\alpha \in \bF^{\times}}\widehat{\Phi}(\alpha t) \Big)  \chi^{-1}(t) \norm[t]_{\A}^{1-s} \ud^{\times}t = \Zeta(1-s, \widehat{\Phi}, \chi^{-1}), \nonumber
\end{align}
\begin{equation} \label{TateIntD}
	\Zeta(s, \Phi, \chi) = L(s,\chi) \cdot \prod_{v \mid \infty} \Zeta_v(s, \Phi_v, \chi_v) \cdot \prod_{\vp < \infty} \frac{\Zeta_\vp(s, \Phi_\vp, \chi_\vp)}{L_{\vp}(s,\chi_{\vp})},
\end{equation}
	together with the \emph{local theory} for the corresponding local zeta integrals (meromorphic continuation and local functional equations)
	$$ \Zeta_v(s, \Phi_v, \chi_v) := \int_{\bF_v^{\times}} \Phi_v(t) \chi_v(t) \norm[t]_v^s \ud^{\times}t. $$	
	Here $\delta_*$ is equal to $1$ if $*$ is satisfied, and $0$ otherwise; $\widehat{\Phi}$ is the $\psi$-Fourier transform defined by
	$$ \widehat{\Phi}(x) := \int_{\A} \Phi(y) \psi(-xy) \ud x; $$
	and the product in (\ref{TateIntD}) is in fact finite because all but finitely many factors are equal to $1$. Consequently, both the (tempered) distribution $\Phi \mapsto \Zeta(s, \Phi, \chi)$ and the \emph{weight distribution}
\begin{equation} \label{TateWt}
	\Phi \mapsto \Wt(s, \Phi, \chi) := \prod_{v \mid \infty} \Zeta_v(s, \Phi_v, \chi_v) \cdot \prod_{\vp < \infty} \frac{\Zeta_\vp(s, \Phi_\vp, \chi_\vp)}{L_{\vp}(s,\chi_{\vp})}
\end{equation}
have meromorphic continuations to $s \in \C$ and functional equations relating $s$ and $1-s$, hence the desired meromorphic continuation and functional equation for $L(s,\chi)$ follows directly. Note that the equation
$$ \Zeta(\tfrac{1}{2},\Phi,\chi) = L(\tfrac{1}{2},\chi) \cdot \Wt(\tfrac{1}{2},\Phi,\chi) $$
realizes the integral representation $\Zeta(\frac{1}{2},\Phi,\chi)$ as a weighted special $L$-value, in the flavour of the moment problems for $L$-functions. The weight is determined by $\Phi$, which is regarded as a \emph{test function}.
	
	Weil \cite{We65} observed that both distributions $\Phi \mapsto \Zeta(s, \Phi, \chi) $ and $\Phi \mapsto \Wt(s, \Phi, \chi) $ are elements in $\Hom_{\A^{\times}}(\Sch(\A), \chi^{-1} \norm_{\A}^{-s})$, which has dimension at most $1$ since locally at each place $\Hom_{\bF_v^{\times}}(\Sch(\bF_v), \chi_v^{-1} \norm_v^{-s})$ has dimension $1$ for any $s \in \C$ by the theory of homogeneous distributions. This explains the local functional equations in a way which is more conceptual than Tate's original proof. The idea of uniqueness of distributions satisfying certain co-variance properties turns out to be insightful and fruitful in the subsequent development of automorphic representation theory (see \cite{Sha74} for example).
	
	
	According to different generalizations of Hecke characters, Tate's thesis has been generalized in many ways, among which the Godement--Jacquet theory \cite{GoJa72} for $\GL_2$ and Rankin--Selberg theory \cite{JPS83} for $\GL_2 \times \GL_1$ are relevant to this paper. Before getting into the theories, we first set up the relevant notation and conventions for $\GL_2$.

	\subsection{Notation for $\GL_2$}
	
	For $R \in \left\{ \bF_v \ \middle| \ v \in V_{\bF} \right\} \cup \{ \A \}$, we define the following subgroups of $\GL_2(R)$
	$$ \gp{Z}(R) = \left\{ z(u) := \begin{pmatrix} u & 0 \\ 0 & u \end{pmatrix} \ \middle| \ u \in R^{\times} \right\}, \quad \gp{N}(R) = \left\{ n(x) := \begin{pmatrix} 1 & x \\ 0 & 1 \end{pmatrix} \ \middle| \ x \in R \right\}, $$
	$$ \gp{A}(R) = \left\{ a(y) := \begin{pmatrix} y & 0 \\ 0 & 1 \end{pmatrix} \ \middle| \ y \in R^{\times} \right\}, \quad \gp{A}(R)\gp{Z}(R) = \left\{ d(t_1,t_2) := \begin{pmatrix} t_1 & \\ & t_2 \end{pmatrix} \ \middle| \ t_1,t_2 \in R^{\times} \right\}, $$
and equip them with the Haar measures on $R^{\times}, R, R^{\times}, R^{\times} \times R^{\times}$ respectively. We reserve
	$$ w = \begin{pmatrix} 0 & 1 \\ -1 & 0 \end{pmatrix} $$
for the Weyl element in $\GL_2(R)$. The product $\gp{B} := \gp{Z} \gp{N} \gp{A}$ is a Borel subgroup of $\GL_2$. We pick the standard maximal compact subgroup $\gp{K} =\prod_v \gp{K}_v$ of $\GL_2(\ag{A})$ by defining
$$ \gp{K}_v
=\begin{cases} \SO_2(\ag{R}), \ \ &\text{if } \bF_v = \ag{R},\\
\SU_2(\C), & \text{if } \bF_v = \C,
\\ \GL_2(\vo_{\vp}), & \text{if } v = \vp < \infty,\end{cases}$$
and equip it with the Haar probability measure $\ud\kappa_v$. Note that at $v \mid \infty$, this measure coincides with
	$$ \ud g = \frac{\ud X}{\norm[\det X]^2} =  \frac{\ud x_1\ud x_2\ud x_3\ud x_4}{\norm[x_1x_4-x_2x_3]^2}, \quad g = X = \begin{pmatrix} x_1 & x_2 \\ x_3 & x_4 \end{pmatrix} \in \GL_2(\R). $$
	We then define and equip the quotient space
	$$ [\PGL_2] := \gp{Z}(\ag{A}) \GL_2(\bF) \backslash \GL_2(\ag{A}) = \PGL_2(\bF) \backslash \PGL_2(\ag{A}) $$
with the product measure $\ud\bar{g} := \prod_v\ud\bar{g}_v$ on $\PGL_2(\ag{A})$ quotient by the discrete measure on $\PGL_2(\bF)$.

Let $\intL^2(\PGL_2)$ denote the (Hilbert) space of Borel measurable functions $\varphi$ satisfying
\begin{align*}
\begin{cases}
\varphi(z \gamma g) = \varphi(g), \quad \text{for all } \gamma \in \GL_2(\bF), ~z \in \gp{Z}(\ag{A}),~ g \in \GL_2(\ag{A}), \\ 
\text{The Petersson norm } \Pairing{\varphi}{\varphi} := \displaystyle\int_{[\PGL_2]} \norm[\varphi(g)]^2 \ud\bar{g} < \infty. 
\end{cases}
\end{align*}	
Let $\intL_0^2(\PGL_2)$ denote the subspace of $\varphi \in \intL^2(\PGL_2)$ such that its \emph{constant term} vanishes:
$$ \varphi_{\gp{N}}(g) := \int_{\bF \backslash \ag{A}} \varphi(n(x)g) \ud x = 0, \quad \text{a.e. } \bar{g} \in [\PGL_2]. $$
It can be shown that $\intL_0^2(\PGL_2)$ is a closed subspace of $\intL^2(\PGL_2)$. The group $\GL_2(\ag{A})$ acts on $\intL_0^2(\PGL_2)$ resp. $\intL^2(\PGL_2)$, giving rise to a unitary representation $\rpR_0$ resp. $\rpR$. The ortho-complement of $\rpR_0$ in $\rpR$ is the orthogonal sum of the one-dimensional spaces
	$$ \ag{C} \left( \xi \circ \det \right) : \quad \xi \text{ a Hecke character such that } \xi^2 = \id $$
and $\rpR_c$, which can be identified as a direct integral representation over the unitary dual of $\bF^{\times} \backslash \ag{A}^{\times} \simeq \ag{R}_+ \times (\bF^{\times} \backslash \ag{A}^{(1)} )$. Precisely, for $\tau \in \ag{R}$ and a unitary character $\chi$ of $\bF^{\times} \backslash \ag{A}^{(1)}$ which is regarded as a unitary character of $\bF^{\times} \backslash \ag{A}^{\times}$ via trivial extension, we associate a unitary representation $\pi(\chi,i\tau)$ of $\GL_2(\ag{A})$ on the following Hilbert space $V_{\chi,i\tau}$ of functions via right regular translation
\begin{align*}
\begin{cases}
f\Big(\Big(\begin{matrix} t_1 & x \\ 0 & t_2 \end{matrix}\Big) g \Big) = \chi(t_1/t_2) |\frac{t_1}{t_2}|_{\ag{A}}^{\frac{1}{2}+i\tau} f(g), \quad \text{for all }t_1,t_2 \in \ag{A}^{\times}, ~x \in \ag{A}, ~g \in \GL_2(\ag{A}) ; \\ 
\text{The induced norm } \Pairing{f}{f} := \displaystyle\int_{\gp{K}} \norm[f(\kappa)]^2 \ud\kappa < \infty . 
\end{cases}
\end{align*}	
	If $f_{i\tau} \in V_{\chi,i\tau}$ satisfies $f_{i\tau} \mid_{\gp{K}} =: h$ is independent of $\tau$, we call it a {\it flat section}. It extends to a holomorphic section $f_s \in \pi(\chi,s)$ for $s \in \C$. Then $\pi(\chi,i\tau)$ is realized as a component of $\rpR_c$ via the Eisenstein series
	$$ \eis(s,h)(g) = \eis(f_s)(g) := \sum_{\gamma \in \gp{B}(\bF) \backslash \PGL_2(\bF)} f_s(\gamma g), $$
which is absolutely convergent for $\Re s > 1/2$ and admits a meromorphic continuation regular at $s=i\tau$.
	
The irreducible components of $\rpR_0$ and $\rpR_c$, both denoted by $\pi$, are called {\it cuspidal} and {\it continuous} automorphic representations, respectively. Let $e_2 \in V_{\pi}^{\infty}$, the subspace of the underlying Hilbert space of $\pi$ consisting of smooth vectors, and let $e_1 \in V_{\pi^{\vee}}^{\infty}$ be an element of the smooth dual space of $V_{\pi}^{\infty}$. The associated function on $\GL_2(\A)$
	$$ \beta(g) = \beta(e_2,e_1)(g) := \Pairing{\pi(g).e_2}{e_1} $$
	is called a (smooth) matrix coefficient of $\pi$. Hence the function $\widecheck{\beta}(g) := \beta(g^{-1})$ is a (smooth) matrix coefficient of the contragredient representation $\pi^{\vee}$. Note that if $\pi$ is an irreducible component of $\rpR_c$, constructed from elements in $\pi(\chi,s)$, the underlying Hilbert space is defined via the induced norm. If $e_{1,s} \in \pi(\chi^{-1},s)$, $e_{2,s} \in \pi(\chi,s)$ are flat sections based on $e_j = e_{j,0} \in \pi(\chi,0) =: \pi(\chi)$, we write the matrix coefficient as
	$$ \beta_s(e_2,e_1)(g) = \beta(e_{2,s}, e_{1,-s})(g) = \Pairing{\pi(\chi,s)(g).e_{2,s}}{e_{1,-s}}. $$
	The pairing actually makes sense for all $s \in \C$, because $\pi(\chi^{-1},-s)$ is naturally the dual of $\pi(\chi,s)$ via the extended induced pairing
	$$ \Pairing{e_{2,s}}{e_{1,-s}} = \int_{\gp{K}} e_{2,s}(\kappa) e_{1,-s}(\kappa) \ud\kappa. $$
	
	At a finite place $\vp$ and $n \in \Z_{\geq 0}$, we introduce the subgroups of $\gp{K}_{\vp}$
	$$ \gp{K}_0[\vp^n] = \left\{ \begin{pmatrix} a & b \\ c & d \end{pmatrix} \in \gp{K}_{\vp} \ \middle| \ c \in \vp^n \right\}, \quad \gp{K}_1[\vp^n] = \left\{ \begin{pmatrix} a & b \\ c & d \end{pmatrix} \in \gp{K}_{\vp} \ \middle| \ c,d-1 \in \vp^n \right\}. $$
	For an admissible irreducible representation $\pi$ of $\GL_2(\bF_{\vp})$, the conductor exponent $\cond(\pi)$ is the least $n \in \Z_{\geq 0}$ such that $\pi$ contains a non-zero vector invariant by $\gp{K}_1[\vp^n]$. The conductor is therefore defined by $\Cond(\pi)=\Nr(\vp)^{\cond(\pi)}$.

	\subsection{Godement--Jacquet theory}
	
	Let $\pi$ be a cuspidal automorphic representation of $\PGL_2(\A)$. The tensor product theorem shows that $\pi \simeq \otimes_v' \pi_v$ is decomposable as a restricted tensor product of local representations $\pi_v$ of $\PGL_2(\bF_v)$. Hence there are decomposable (smooth) vectors in the underlying space $V_{\pi}^{\infty}$, which span a dense subspace. If $e_1=\otimes_v' e_{1,v} \in V_{\pi^{\vee}}^{\infty}$ and $e_2 = \otimes_v' e_{2,v} \in V_{\pi}^{\infty}$ are decomposable vectors, then so is the associated matrix coefficient $\beta = \otimes_v' \beta_v$, just behaving like a Hecke character. The decomposable matrix coefficients are good analogues of Hecke characters, for which one establishes the fundamental equations with (decomposable) Bruhat--Schwartz functions $\Psi = \otimes_v' \Psi_v \in \Sch(\Mat_2(\A))$
\begin{equation}\label{GJIntG}
\begin{split}
	\Zeta(s, \Psi ,\beta) &:= \int_{\GL_2(\A)} \Psi(g) \beta(g) \norm[\det g]_{\A}^{s+1/2} \ud^{\times}g \\
	&= \int_1^{\infty} \left[ \int_{(\GL_2(\bF) \backslash G^1)^2} \sum_{\xi \in \GL_2(\bF)} \Psi(h_2^{-1} \xi z(t) h_1) \cdot e_1(h_1) e2(h_2) \ud^{\times}h_1 \ud^{\times}h_2 \right] t^{2s+1} \ud^{\times}t \\
	&\quad + \int_1^{\infty} \left[ \int_{(\GL_2(\bF) \backslash G^1)^2} \sum_{\xi \in \GL_2(\bF)} \widehat{\Psi}(h_1^{-1} \xi z(t) h_2) \cdot e_1(h_1) e_2(h_2) \ud^{\times}h_1 \ud^{\times}h_2 \right] t^{3-2s} \ud^{\times}t  \\
	&= \int_{\GL_2(\A)} \widehat{\Psi}(g) \widecheck{\beta}(g) \norm[\det g]_{\A}^{3/2-s} \ud^{\times}g = \Zeta(1-s, \widehat{\Psi}, \widecheck{\beta}), 
\end{split}
\end{equation}
\begin{equation} \label{GJIntD}
	\Zeta(s, \Psi, \beta) = L(s,\pi) \cdot \prod_{v \mid \infty} \Zeta_v(s, \Psi_v, \beta_v) \prod_{\vp < \infty} \frac{\Zeta_{\vp}(s, \Psi_{\vp}, \beta_{\vp})}{L_{\vp}(s,\pi_{\vp})},
\end{equation}
	together with the \emph{local theory} for the corresponding local zeta integrals (meromorphic continuation and local functional equations)
	$$ \Zeta_v(s, \Psi_v, \beta_v) := \int_{\GL_2(\bF_v)} \Psi_v(g) \beta_v(g) \norm[\det g]_v^{s+1/2} \ud^{\times}g. $$	
	Here we have defined the group $G^1 := \left\{ g \in \GL_2(\A) \ \middle| \ \norm[\det(g)]_{\A}=1 \right\}$ so that $\GL_2(\bF) \backslash G^1$ is compact, and the $\psi$-Fourier transform $\widehat{\Phi}$ is defined by
	$$ \widehat{\Phi}(x) := \int_{\Mat_2(\A)} \Phi(y) \psi(-\Tr(xy)) \ud x, $$
	viewing $\Mat_2(\A)$ simply as $\A^4$. From this point, it is clear that the Godement--Jacquet theory is a precise generalization of Tate's thesis, which preserves the local-global nature. In particular, the corresponding weight distribution (analogue of (\ref{TateWt}))
\begin{equation} \label{GoJWt}
	\Psi \mapsto \Wt(s, \Psi, \beta) := \prod_{v \mid \infty} \Zeta_v(s, \Psi_v, \beta_v) \prod_{\vp < \infty} \frac{\Zeta_{\vp}(s, \Psi_{\vp}, \beta_{\vp})}{L_{\vp}(s,\pi_{\vp})}
\end{equation}
	realizes Godement--Jacquet's zeta integrals as weighted $L$-functions for automorphic representations of $\GL_2(\A)$. The only thing left unclear for an exact analogue is the co-variance property of these tempered distributions along Weil's idea. This is made clear in the subsequent development of the Rankin--Selberg theory, which we will recall in the next subsection.

	\subsection{Rankin--Selberg theory}
	
	The key to find a theory of zeta integrals with co-variant distributions is the uniqueness of the Whittaker functional. 
\begin{definition}
	At $v \in V_{\bF}$, a $\psi_v$-Whittaker functional is a continuous functional $\ell_v$ on $V_{\pi_v}^{\infty}$ satisfying
	$$ \ell_v(\pi_v(n(x)).f) = \psi_v(x) \ell_v(f), \quad \forall f \in V_{\pi_v}^{\infty}, x \in \bF_v. $$
	A $\psi$-Whittaker functional is a continuous functional $\ell$ on $V_{\pi}^{\infty}$ satisfying
	$$ \ell(\pi(n(x)).f) = \psi(x) \ell(f), \quad \forall f \in V_{\pi}^{\infty}, x \in \A. $$
\end{definition}

\noindent It turns out that for every automorphic representation $\pi$, the Whittaker functional exists and is unique up to a scalar both locally and globally. This allows one to define the Whittaker function associated to a smooth vector $f$ by
\begin{align*}
W_f(g)=
\begin{cases}
\ell_v(\pi_v(g).f), \ \ & f\in V_{\pi_v}^{\infty},\\
\ell(\pi(g).f), & f \in V_{\pi}^{\infty}.
\end{cases}
\end{align*}

\begin{definition}
	As $f$ traverses $V_{\pi_v}^{\infty}$ (resp. $V_{\pi}^{\infty}$), the functions $g \mapsto W_f(g)$ form a subspace of 
	$$ \Cont^{\infty}(\gp{N}(\F_v) \backslash \GL_2(\F_v), \psi_v) := \left\{ W \in \Cont^{\infty}(\GL_2(\F_v)) \ \middle| \ W(n(x)gz(u)) = \psi_v(x) W(g), \forall x \in \F_v, u \in \F_v^{\times} \right\} $$	
	(resp. similarly defined $\Cont^{\infty}(\gp{N}(\A) \backslash \GL_2(\A), \psi)$). This subspace, equipped with the multiplication from right by $\GL_2(\bF_v)$ (resp. $\GL_2(\A)$) is isomorphic to $V_{\pi_v}^{\infty}$ (resp. $V_{\pi}^{\infty}$), and is called the $\psi_v$(resp. $\psi$)-Whittaker model of $\pi_v$ (resp. $\pi$), denoted by $\Whi(\pi_v, \psi_v)$ (resp. $\Whi(\pi,\psi)$).
\end{definition}

\begin{remark}
	For $R = \bF_v$ or $\A$, the map of restricting functions on $\GL_2(R)$ to $\gp{A}(R)$ is injective on the Whittaker models. This is the Kirilov conjecture, proved by Baruch \cite{Ba03}. The Kirillov models are
	$$ \Kir(\pi_v, \psi_v) = \left\{ K: \bF_v^{\times} \to \C \ \middle| \ K(t) = W(a(t)) \text{ for some } W \in \Whi(\pi_v, \psi_v) \right\}, $$
	$$ \Kir(\pi, \psi) = \left\{ K: \A^{\times} \to \C \ \middle| \ K(t) = W(a(t)) \text{ for some } W \in \Whi(\pi, \psi) \right\}. $$
	The local Kirillov models are equipped with natural $\GL_2(\bF_v)$-invariant pairings given by
	$$ \Pairing{K}{K} = \int_{\bF_v^{\times}} \norm[W(t)]^2 \ud^{\times}t. $$
	We call $K(t) = W(a(t)) = W_f(a(t))$ the Kirillov function of $W$ or $f$.
\end{remark}

Let $\varphi \in V_{\pi}^{\infty}$ be an automorphic form in an automorphic representation $\pi$. Its $\psi$-Whittaker function is given by an integral over a compact domain
	$$ W(g) = W_{\varphi}(g) = \int_{\bF \backslash \A} \varphi(n(x)g) \psi(-x) \ud x. $$
We deduce the Fourier-Whittaker expansion
	$$ \varphi(g) - \varphi_{\gp{N}}(g) = \sum_{\alpha \in \bF^{\times}} W(a(\alpha)g). $$
Suppose $\varphi$ is decomposable, then so is $W = \otimes_v' W_v$. Let $\chi = \otimes_v' \chi_v$ be a Hecke character. We assume $\pi$ to be cuspidal for simplicity. The slight generalization to the Eisenstein case is well-known to experts, and can be found in \cite[\S 2.1]{Wu19_S} as a special example in the theory of regularized integrals. We have the fundamental equations for the Rankin--Selberg zeta integrals

\begin{align}
	\Zeta(s, \varphi ,\chi) &:= \int_{\A^{\times}} W(a(t)) \chi(t) \norm[t]_{\A}^{s-1/2} \ud^{\times}t \nonumber \\
	&= \int_{\substack{t \in \A^{\times} \\ \norm[t]_{\A} \geq 1}} \varphi(a(t)) \chi(t) \norm[t]_{\A}^{s-1/2} \ud^{\times}t + \int_{\substack{t \in \A^{\times} \\ \norm[t]_{\A} \geq 1}} \varphi(a(t)w) \chi^{-1}(t) \norm[t]_{\A}^{1/2-s} \ud^{\times}t \label{RSIntG} \\
	&= \int_{\A^{\times}} W(a(t)w) \chi^{-1}(t) \norm[t]_{\A}^{1/2-s} \ud^{\times}t = \Zeta(1-s, \pi(w).\varphi, \chi^{-1}), \nonumber
\end{align}
\begin{equation} \label{RSIntD}
	\Zeta(s, \varphi, \chi) = L(s,\pi \times \chi) \cdot \prod_{v \mid \infty} \Zeta_v(s, W_v, \chi_v) \prod_{\vp < \infty} \frac{\Zeta_{\vp}(s, W_{\vp}, \chi_{\vp})}{L_{\vp}(s,\pi_{\vp} \times \chi_{\vp})},
\end{equation}
	together with the \emph{local theory} for the corresponding local zeta integrals (meromorphic continuation and local functional equations)
	$$ \Zeta_v(s, W_v, \chi_v) := \int_{\bF_v^{\times}} W_v(a(t)) \chi_v(t) \norm[t]_v^{s-1/2} \ud^{\times}t. $$
	Both $\varphi \mapsto \Zeta(s,\varphi,\chi)$ and the weight function
\begin{equation} \label{RSWt}
	\varphi \mapsto \Wt(s,\varphi,\chi) := \prod_{v \mid \infty} \Zeta_v(s, W_v, \chi_v) \prod_{\vp < \infty} \frac{\Zeta_{\vp}(s, W_{\vp}, \chi_{\vp})}{L_{\vp}(s,\pi_{\vp} \times \chi_{\vp})}
\end{equation}
	are continuous functionals in the one dimensional space $\Hom_{\A^{\times}}(V_{\pi}^{\infty}, \chi^{-1} \norm_{\A}^{-s})$, where $\A^{\times}$ is regarded as the subgroup $\gp{A}(\A)$ of $\GL_2(\A^{\times})$. If $\chi = \id$ is the trivial character, we always omit it from the notation, in which case it is expected that $L(s,\pi \times \id) = L(s, \pi)$ coincides with the $L$-function defined via Godement--Jacquet zeta integrals. In general, one expects $L(s, \pi \times \chi) = L(s, \pi \otimes \chi)$ as defined in the Godement--Jacquet theory for $\GL_2$ instead of $\PGL_2$. This is explained in the following remark.
	
\begin{remark}
	Intuitively, the Whittaker functions are ``generalized'' matrix coefficients by letting $e_1 \in V_{\pi_v^{\vee}}^{\infty}$ (resp. $V_{\pi^{\vee}}^{\infty}$) tend to $\ell_v$ (resp. $\ell$) in $\beta_v(e_1,w)$ (resp. $\beta(e_1,w)$). Hence it is reasonable to believe that, as $f$ traverses $V_{\pi}^{\infty}$ and $\Psi$ traverses $\Sch(\Mat_2(\A))$, the collection of the integrals
\begin{equation} \label{GoJIntVar}
	\Zeta(s, \Psi, W_f) := \int_{\GL_2(\A)} \Psi(g) W_f(g) \norm[\det g]_{\A}^{s+1/2} \ud g 
\end{equation}
	is essentially the same as the collection of the Godement--Jacquet zeta integrals. A tricky computation shows that the integral (\ref{GoJIntVar}) is the same as $\Zeta(s, \varphi, \chi)$ defined by (\ref{RSIntG}) for a smooth $\varphi \in V_{\pi}^{\infty}$. The details can be found in \cite[\S 3 \& 4 \& 9 \& 11]{JPS79}. It shows the equivalence between the Godement--Jacquet theory for $\GL_n$ and Rankin--Selberg theory for $\GL_n \times \GL_1$.
\end{remark}

\begin{remark}
	Note that $\Zeta_v(s, W_v, \chi_v)$ (resp. $\Zeta_v(s, W_v)$) depends only on the Kirillov function $K_v$ associated with $W_v$. We shall write it as $\Zeta_v(s, K_v, \chi_v)$ (resp. $\Zeta_v(s, K_v)$) as well.
\end{remark}

	The collected zeta integrals, as well as the corresponding local theories, provide powerful tools of meromorphic continuations.

\section{Motohashi's Formula: A Distributional Version}
\label{sec:Motohashiformula}
	
	Motohashi \cite{Mo93} discovered a beautiful equation relating the fourth moment of the Riemann zeta function to the cubic moment of modular $L$-functions for the full modular group $\SL_2(\Z)$, with certain generalization in \cite{BrM03}. For $w$ a sufficiently nice weight function, the original formula of Motohashi \cite{Mo93} takes the shape
\begin{align*}
	\int_{\R}|\zeta(\tfrac{1}{2}+it)|^4 w(t) \ud t=\sum_f L(\tfrac{1}{2},f)^3 \widecheck{w}(t_f) + (\textrm{CSC}),
\end{align*}
where the sum runs over all holomorphic/Maass forms $f$ with spectral parameter $t_f$ for the group $\SL_2(\Z)$, and $\widecheck{w}$ is a certain integral transform of $w$ given explicitly in terms of hypergeometric functions. The term (CSC) is the analogous contribution from Eisenstein series.
	
	We obtained a version of Motohashi's formula in \cite{Wu22}, in which the weighted $L$-values are realized as some integral representations in the same flavour as those recalled in the previous section. In particular, the weight functions $w(t)$ and $\widecheck{w}(t_f)$ are replaced by the corresponding weight distributions.
	
	Precisely, let $\Psi = \otimes_v' \Psi_v \in \Sch(\Mat_2(\A))$ be decomposable. Let $S$ be a set of places of $\bF$ containing the set $S_{\infty}$ of all archimedean places and $\vp \mid \Dis_{\bF}$, such that the test function $\Psi = \otimes_v' \Psi_v$ is decomposable, and 
\begin{itemize}
	\item $\Psi_v \in \Sch(\GL_2(\bF_v))$ is a Schwartz function on the group of invertible elements at each $v \mid \infty$,
	\item $\Psi_{\vp} = \id_{\Mat_2(\vo_{\vp})}$ for $\vp \notin S$.
\end{itemize}
	In what follows, $\Bas(\sigma)$ will denote any orthogonal basis of the underlying Hilbert space of a representation $\sigma$, which consist of smooth vectors.

	For any cuspidal automorphic representation $\pi = \otimes_v' \pi_v$ of $\PGL_2(\A)$, let $\Bas(\pi)$ be any orthogonal basis of $V_{\pi}^{\infty} \subset \intL_0^2([\PGL_2])$. The Petersson norm gives an isomorphism
	$$ V_{\pi}^{\infty} \to V_{\pi^{\vee}}^{\infty}, \quad \varphi \mapsto \left( \ell_{\varphi}: f \mapsto \Pairing{f}{\varphi} \right). $$
	Hence the set $\Bas(\pi^{\vee}) = \{ e^{\vee} := \Pairing{e}{e}^{-1} \overline{e} ~| ~e \in \Bas(\pi) \}$ can be naturally viewed as the dual basis of $\Bas(\pi)$ in $V_{\pi^{\vee}}^{\infty}$. We define an integral representation of $L(1/2,\pi)^3$
\begin{equation} \label{3rdMCuspDef}
	\Mt_3(\Psi \mid \pi) := \sum_{e_1,e_2 \in \Bas(\pi)} \Zeta( \tfrac{1}{2}, \Psi, \beta(e_2,e_1^{\vee})) \cdot \Zeta( \tfrac{1}{2}, e_1) \cdot \Zeta( \tfrac{1}{2}, e_2^{\vee}).
\end{equation}
	The local version at $v$ is similarly defined in the Kirillov model $\Kir(\pi_v, \psi_v)$ as
\begin{equation} \label{3rdMCuspLocDef}
	\Mt_{3,v}(\Psi_v \mid \pi_v) = \sum_{K_1,K_2 \in \Bas(\Kir(\pi_v,\psi_v))} \Zeta( \tfrac{1}{2}, \Psi_v, \beta(K_2, K_1^{\vee}) \cdot \Zeta \left( \tfrac{1}{2}, K_1 \right) \cdot \Zeta( \tfrac{1}{2}, K_2^{\vee}),
\end{equation}
	which give the relevant weight distribution
\begin{equation} \label{Wt3rdMCusp}
	\Wt_3(\Psi \mid \pi) = \prod_{v \text{ real}} \Mt_{3,v}(\Psi_v \mid \pi_v) \prod_{\vp \in S-S_{\infty}} \Mt_{3,\vp}(\Psi_{\vp} \mid \pi_{\vp}) \frac{L( 1, \pi_{\vp} \times \pi_{\vp})}{L( \tfrac{1}{2}, \pi_{\vp})^3}.
\end{equation}
	They are related by
\begin{equation} \label{CubWtCuspD}
	\Mt_3(\Psi \mid \pi) = \frac{L( \tfrac{1}{2},\pi)^3}{2 \Lambda_{\bF}(2) L(1,\pi,\mathrm{Ad})} \cdot \Wt_3(\Psi \mid \pi).
\end{equation}
	Here, $\Lambda_{\bF}(s)$ is the completed Dedekind zeta function, $L(1, \pi_{\vp} \times \pi_{\vp})$ is the local component of Rankin-Selberg $L$-functions for $\GL_2 \times \GL_2$, and $L(s,\pi,\mathrm{Ad})$ is $L$-function of the adjoint lift of $\pi$ to $\GL_3(\A)$. 
\begin{remark}	
	Note that the values $ L(1, \pi_{\vp} \times \pi_{\vp}), L(\tfrac{1}{2}, \pi_{\vp})$ are positive, and are $\asymp 1$ uniformly in $\Nr(\vp)$. Hence they can be ignored for the purpose of this paper.
\end{remark}

	Let $\pi$ be a continuous automorphic representation, whose elements are Eisenstein series constructed from the induced representations $\pi(\chi,s)$, parametrized by a unitary Hecke character $\chi$ and $s \in i\R$. Recall that the underlying Hilbert space $V_{\pi}$ is a induced model of $\pi(\chi,s)$. We define
\begin{equation} \label{3rdMEisDef}
	\Mt_3(\Psi \mid \chi,s) := \sum_{e_1,e_2 \in \Bas(\pi(\chi))} \Zeta( \tfrac{1}{2}, \Psi, \beta_s(e_2,e_1^{\vee})) \cdot \Zeta( \tfrac{1}{2}, e_{1,s}) \cdot \Zeta( \tfrac{1}{2}, e_{2,-s}^{\vee}).
\end{equation}
	At a place $v$, we write $K_{j,s}$ (resp. $K_{j,s}^{\vee}$) for the Whittaker function of the flat section $e_{j,s}$ (resp. $e_{j,s}^{\vee}$). The local version of (\ref{3rdMEisDef}) is given by
\begin{equation} \label{3rdMEisLocDef}
	\Mt_{3,v}(\Psi_v \mid \chi_v, s) = \sum_{e_1,e_2 \in \Bas(\pi(\chi_v))} \Zeta( \tfrac{1}{2}, \Psi_v, \beta_s(e_2, e_1^{\vee})) \cdot \Zeta( \tfrac{1}{2}, K_{1,s}) \cdot \Zeta( \tfrac{1}{2}, K_{2,-s}^{\vee}),
\end{equation}
	which gives the relevant weight distribution
\begin{equation} \label{Wt3rdMEis}
	\Wt_3(\Psi \mid \chi,s) = \prod_{v \text{ real}} \Mt_{3,v}(\Psi_v \mid \chi_v,s) \prod_{\vp \in S-S_{\infty}} \Mt_{3,\vp}(\Psi_{\vp} \mid \chi_{\vp},s) \frac{L( 1+2s, \chi_{\vp}^2) L( 1-2s, \chi_{\vp}^{-2})}{L ( \tfrac{1}{2}+s, \chi_{\vp})^3 L( \tfrac{1}{2}-s, \chi_{\vp})^3}.
\end{equation}
	They are related by
\begin{equation} \label{CubWtEisD}
	\Mt_3(\Psi \mid \chi, s) = \frac{L( \tfrac{1}{2}+s, \chi)^3 L( \tfrac{1}{2}-s, \chi^{-1})^3}{L(1+2s,\chi^2)L(1-2s, \chi^{-2})} \cdot \Wt_3(\Psi \mid \chi,s).
\end{equation}
	
	For a Hecke character $\chi = \otimes_v' \chi_v$, the integral representation of the fourth moment is a $4$-dimensional Tate's integral
\begin{equation} \label{4thMSinDef}
	\Mt_4(\Psi \mid \chi,s)  = \int_{(\A^{\times})^4} \Psi\Big(\Big( \begin{matrix} x_1 & x_2 \\ x_3 & x_4 \end{matrix} \Big)\Big) \chi \Big( \frac{x_1x_4}{x_2x_3} \Big)  \norm[x_1 x_4]_{\A}^{s} \norm[x_2 x_3]_{\A}^{1-s} \ud^{\times}x_1\ud^{\times}x_2\ud^{\times}x_3\ud^{\times}x_4.
\end{equation}
	The local version is defined by
\begin{equation} \label{4thMLocDef}
	\Mt_{4,v}(\Psi_v \mid \chi_v,s)  = \int_{(\bF_v^{\times})^4} \Psi_v\Big(\Big( \begin{matrix} x_1 & x_2 \\ x_3 & x_4 \end{matrix} \Big)\Big) \chi_v\Big( \frac{x_1x_4}{x_2x_3} \Big)\norm[x_1 x_4]_v^{s} \norm[x_2 x_3]_v^{1-s}  \ud^{\times}x_1\ud^{\times}x_2\ud^{\times}x_3\ud^{\times}x_4,
\end{equation}
	which give the relevant weight distribution
\begin{equation} \label{Wt4thM}
	\Wt_4(\Psi \mid \chi,s) = \prod_{v \text{ real}} \Mt_{4,v}(\Psi_v \mid \chi_v,s) \prod_{\vp \in S-S_{\infty}} \frac{\Mt_{4,\vp}(\Psi_{\vp} \mid \chi_{\vp},s)}{L( \tfrac{1}{2}+s, \chi_{\vp})^2 L( \tfrac{1}{2}-s, \chi_{\vp})^2}.
\end{equation}
	They are related by
\begin{equation} \label{4thMD}
	\Mt_4(\Psi \mid \chi, s) = L( \tfrac{1}{2}+s, \chi)^2 L( \tfrac{1}{2}-s, \chi^{-1})^2 \cdot \Wt_4(\Psi \mid \chi,s).
\end{equation}

\begin{remark}
	Note that the global (resp. local) distributions introduced in the last two paragraphs depend only on $\chi \norm_{\A}^s$ (resp. $\chi_v \norm_v^s$), instead of $(\chi,s)$ (resp. $(\chi_v,s)$) as two separate variables.
\end{remark}

	Finally, we define the tempered distributions $\Mt_3$ and $\Mt_4$ mimicking the cubic and fourth moments in Motohashi's original formula as
\begin{equation} \label{3rdMDef}
	\Mt_3(\Psi) := \sum_{\pi \text{ cuspidal}} \Mt_3(\Psi \mid \pi) + \sum_{\chi \in \widehat{\R_+ \bF^{\times} \backslash \A^{\times}}} \int_{-\infty}^{\infty} \Mt_3(\Psi \mid \chi,i\tau) \frac{\ud\tau}{4\pi},
\end{equation}
\begin{equation} \label{4thMDef}
	\Mt_4(\Psi) = \frac{1}{\zeta_{\bF}^*} \sum_{\chi \in \widehat{\R_+ \bF^{\times} \backslash \A^{\times}}} \int_{\Re s = \frac{1}{2}} \Mt_4(\Psi \mid \chi,s) \frac{\ud s}{2\pi i}.
\end{equation}
	
	The distributional version of Motohashi's formula in \cite{Wu22} is summarized as follows.
\begin{theorem}\label{thm:Motohashi}
	Let $\Psi \in \Sch(\Mat_2(\A))$ be a Schwartz function. We have an equation of tempered distributions
\begin{equation}
	\Mt_3(\Psi) + \Dt_3(\Psi) = \Mt_4(\Psi) + \Dt_4(\Psi),
\label{MTF}
\end{equation}
	where the degenerate terms $\Dt_3(\Psi), \Dt_4(\Psi)$ are defined by
	\begin{align}
		\Dt_3(\Psi)&= \frac{1}{\zeta_{\bF}^*} \Res_{s=\frac{1}{2}} \Mt_3(\Psi \mid \id, s),\label{DSF}\\
		\Dt_4(\Psi)&= \frac{1}{\zeta_{\bF}^*} \left\{ \Res_{s=1} \Mt_4(\Psi \mid \mathbbm{1},s) - \Res_{s=0} \Mt_4(\Psi \mid \mathbbm{1},s) \right\}.\label{DGF}
	\end{align}
\end{theorem}

The distributions appearing in the above version of Motohashi's formula are co-variant according to a natural action of three tori $(\A^{\times})^3$ on $\Mat_2(\A)$. Note that the (intermediate) degenerate terms are regrouped with respect to this action as those co-variant distributions supported in the degenerate orbits (up to partial Fourier transforms), just as in Tate's fundamental equation (\ref{TateIntG}): $\Phi(0)$ and $\widehat{\Phi}(0)$ are supported in the degenerate orbit $\{ 0 \}$ for the multiplication of $\bF^{\times}$ on $\bF$, and are elements in $\Hom_{\A^{\times}}(\Sch(\A), \id)$ and $\Hom_{\A^{\times}}(\Sch(\A), \norm_{\A}^{-1})$ respectively.

\section{Local Computations}
\label{sec:localcomputations}

In this section we give the precise constructions of test functions and the relevant estimates required in Motohashi's formula. 
We work locally at a non-archimedean place, hence we will focus on the local field $\bF_\vp$ with the number field $\bF$ and a prime ideal $\vp$. But we omit the subscript $\vp$ for simplicity of notation, so that throughout this section $\bF$ will refer to $\bF_\vp$ with valuation ring $\vo$ and prime ideal $\vp$. We also fix $\varpi$ as a generator of $\vp$ such that $\vp = \varpi \vo$. Denote by $q$ the cardinality of residue field $\vo/\vp.$ All asymptotic relations in this section will be formulated as $q=|\vo/\vp|\rightarrow+\infty$ in the set of rational prime powers.
	Without loss of generality, we may assume the place $\vp$ does not divide $2$, since there are only finitely many places lying above $2$ given the number field, at which any trivial bound of the dual weight functions suffices for Theorem \ref{Main}. 

	The constructions of test functions are based on the following classifications of unitary irreducible representations $\pi$ with conductor exponents $\fa(\pi) \leq 4$. 
	\begin{itemize}
	\item {\it Case $0: \fa(\pi)=0$.} This is the spherical case. We use the test function at an unramified place and need not treat it in this section.
	
	\item{\it Case $1:\fa(\pi)=1$.} In this case, $\pi$ is either the Steinberg representation $\mathrm{St}$, or its twist $\mathrm{St}_{\xi}$ by the unique unramified quadratic character $\xi$. Namely, $\xi \mid_{\vo^{\times}} = \mathbbm{1}$ is trivial and $\xi(\varpi)=-1$.
	
	\item{\it Case $2:\fa(\pi)=2$.} If $\pi$ is not supercuspidal, then $\pi$ is either a twist $\mathrm{St}_{\xi}$ of the Steinberg representation or $\pi \simeq \pi(\xi, \xi^{-1})$, where the character $\xi$ has conductor exponent $\fa(\xi)=1$. If $\pi$ is supercuspidal, then it has depth $0$ and is constructed from a character of the group of invertible elements of the quadratic field extension of the residual field.
	
	\item{\it Case $3:\fa(\pi)=3$.} This pushes $\pi$ to be a \emph{simple supercuspidal representation}, details of whose construction and basic properties can be found in Knightly and Li \cite{KL15} or Luo, Pi and Wu \cite[Appendix A]{LPW23}.
	
	\item{\it Case $4:\fa(\pi)=4$.} Either $\pi$ is supercuspidal, or it is of the form $\pi \simeq \pi(\xi,\xi^{-1})$ for some character $\xi$ with conductor exponent $\cond(\xi)=2$. We exclude the first case in our main theorems (See Remark \ref{rmk:Hu-Petrow-Young}). The second case was already done in \cite{BFW21+}. So we do not need to treat this case.
	\end{itemize}

	For notational clarity, we name the focused representation as $\pi_0$. In each case, we will give an explicit test function $C \in \cC_{\mathrm{c}}^{\infty}(\GL_2(\bF))$ and estimate both the cubic moment weight function $\Mt_3(C \mid \pi)$ and its dual weight function $\Mt_4(C \mid \chi, \tfrac{1}{2})$ on the fourth moment side.

	\subsection{Case 1: $\fa(\pi)=1$}
	\label{sec:Case1}
	
	In this case we write $\pi_0 = \mathrm{St}_{\xi}$ and consider the test function
	$$ C(g) = \frac{1}{\Vol(\gp{K}_0[\vp])} \mathbbm{1}_{\gp{K}_0[\vp]}(g). $$
	
\begin{lemma}
Suppose $C \in \cC_{\mathrm{c}}^{\infty}(\GL_2(\bF))$ is chosen as above.
\begin{enumerate}[$(1)$]
\item For any irreducible admissible representation $\pi$ of $\PGL_2(\bF)$, the operator $\pi(C)$ is zero unless $\fa(\pi) \leq 1,$ in which case we have $\Mt_3(C \mid \pi) \geq 0.$ For unramified $\xi,$ we have 
	$$ \Mt_3(C \mid \mathrm{St}_{\xi}) \asymp 1. $$

\item The dual weight $\Mt_4(C \mid \chi) $ vanishes unless $\fa(\xi) = 0,$ in which case we have
	$$\Mt_4(C \mid \chi, \tfrac{1}{2}) \asymp q^{\frac{1}{2}}. $$
\end{enumerate}
\label{LocMainBd1}
\end{lemma}

\begin{proof}
(1) For any $\pi$, $\pi(C)$ is the orthogonal projection onto the subspace of $\gp{K}_0[\vp]$-invariant vectors. Its non-vanishing implies $\fa(\pi) \leq 1$. Since $C(g) = \overline{C}*C(g)$, where the convolution is defined in $\GL_2(g)$, it is of positive type. The proof of \cite[Lemma 3.2]{BFW21+} applies, and yields the non-negativity of $\Mt_3(C \mid \pi)$ for any $\pi$. For unramified $\xi$, we have $\fa(\mathrm{St}_{\xi})=1$. Let $K_{\xi}$ be the Kirillov function of a new vector in $\mathrm{St}_{\xi}$. We may assume
\begin{align*}
K_{\xi}(t) = \xi(t) \norm[t] \cdot \mathbbm{1}_{\vo}(t).
\end{align*}
Then the first part of the lemma follows from
\begin{align*}
\Mt_3(C \mid \mathrm{St}_{\xi}) 
= \Big|\int_{\bF^{\times}} W_{\xi}(t) \ud^{\times}t\Big|^2 \cdot \Big( \int_{\bF^{\times}} \extnorm{W_{\xi}(t)}^2 \ud^{\times}t \Big)^{-1} = \frac{1+\xi(\varpi)q^{-1}}{1-\xi(\varpi)q^{-1}}
\end{align*}
and $\xi(\varpi) = \pm 1$.
	
\noindent (2) Note that
\begin{align} \label{eq:LocDualWt1}
\Mt_4(\mathbbm{1}_{\gp{K}_0[\vp]} \mid \chi, \tfrac{1}{2}) &= \int_{\big(\begin{smallmatrix} x_1 & x_2 \\ x_3 & x_4 \end{smallmatrix}\big) \in \big(\begin{smallmatrix} \vo^{\times} & \vo \\ \vp & \vo^{\times} \end{smallmatrix}\big)} \chi \Big( \frac{x_1 x_4}{x_2 x_3} \Big)|x_1x_2x_3x_4|^{\frac{1}{2}}\ud^{\times}x_1\ud^{\times}x_2\ud^{\times}x_3\ud^{\times}x_4 \nonumber \\
&= \mathbbm{1}_{\fa(\chi)=0} \cdot \Big(\int_{\vo} \chi(x_2)^{-1} \norm[x_2]^{\frac{1}{2}} \ud^{\times}x_2\Big)\cdot \Big(\int_{\vp} \chi(x_3)^{-1} \norm[x_3]^{\frac{1}{2}} \ud^{\times}x_3\Big) \nonumber \\
&= \mathbbm{1}_{\fa(\chi)=0} \cdot \chi(\varpi)^{-1} q^{-\frac{1}{2}} \cdot (1-\chi(\varpi)^{-1}q^{-\frac{1}{2}})^{-2}.
\end{align}
The lemma follows immediately from $\Vol(\gp{K}_0[\vp])=(1+q)^{-1}$.
\end{proof}

\begin{lemma}\label{LocAuxBd1}
	We have
	$$ \Mt_3(C \mid \id, s) = \frac{2}{(1-q^{-\frac{1}{2}+s})(1-q^{-\frac{1}{2}-s})}, \quad \Mt_4( C \mid \id, \tfrac{1}{2} + s) = \frac{(q+1)q^{s-\frac{1}{2}}}{(1-q^{s-\frac{1}{2}})^2}. $$
\end{lemma}
\begin{proof}
It suffices to prove the first formula since the second one follows from (\ref{eq:LocDualWt1}). For simplicity of notation, we write $\pi(s) = \pi(\norm^s, \norm^{-s})$. Let $e_0,e_1$ form an orthonormal basis of $\pi(0)^{\gp{K}_0[\vp]}$, so that the sets of flat sections $\{ e_0(s), e_1(s) \} \subset \pi(s)^{\gp{K}_0[\vp]}$ and $\{ e_0(-s), e_1(-s) \} \subset \pi(-s) = \pi(s)^{\vee}$ are dual bases to each other. Such a basis exists because the group action of $\pi(0)$ (on the induced model) is defined over $\R$. Let $W_j(s)$, resp. $W_j(-s)$ be the Whittaker function of $e_j(s)$, resp. $e_j(-s)$ with respect to $\psi$, resp. $\overline{\psi}$. Then we have by definition
		$$ \Mt_3(C \mid \id, s) = \Zeta(\tfrac{1}{2}, W_0(s)) \Zeta(\tfrac{1}{2}, W_0(-s)) + \Zeta(\tfrac{1}{2}, W_1(s)) \Zeta(\tfrac{1}{2}, W_1(-s)). $$
		
	Note that $e_0(s)$ and $\pi(s,a(\varpi^{-1}))e_0(s)$ also form a basis of $\pi(s)^{\gp{K}_0[\vp]}$. We introduce the matrix of base change
	$$ (e_0(s),~ e_1(s)) = (e_0(s), ~\pi(s,a(\varpi^{-1}))e_0(s)) P(s). $$
	Then we have the relation
	$$ P(s)^T \begin{pmatrix} 1 & \sigma_1(s) \\ \sigma_1(s) & 1 \end{pmatrix} P(-s) = I \quad \Longrightarrow \quad P(s)P(-s)^T = \begin{pmatrix} 1 & \sigma_1(s) \\ \sigma_1(s) & 1 \end{pmatrix}^*,$$
	where $g^* := (g^T)^{-1}$ and $\sigma_1(s) = \Pairing{\pi(s,a(\varpi))e_0(s)}{e_0(-s)}$ is given by the Macdonald formula \cite[Theorem 4.6.6]{Bu98} as
	$$ \sigma_1(s) = \frac{q^{-\frac{1}{2}}}{1+q^{-1}} (q^s + q^{-s}). $$
	By linearity, we infer
\begin{align*}
	\Mt_3(C \mid \id,s) &=(\Zeta(\tfrac{1}{2}, W_0(s)),~\Zeta( \tfrac{1}{2}, W_1(s)))(\Zeta( \tfrac{1}{2}, W_0(-s)),~\Zeta ( \tfrac{1}{2}, W_1(-s)) )^T\\
	&=(\Zeta(\tfrac{1}{2}, W_0(s)),~\Zeta( \tfrac{1}{2}, W_0(s))) P(s) P(-s)^T (\Zeta(\tfrac{1}{2}, W_0(-s)),~\Zeta( \tfrac{1}{2}, W_0(-s)))^T\\
	&= \Zeta(\tfrac{1}{2}, W_0(s)) \Zeta( \tfrac{1}{2}, W_0(-s)) \cdot (1,~1) \begin{pmatrix} 1 & \sigma_1(s) \\ \sigma_1(s) & 1 \end{pmatrix}^* (1,~1)^T.
\end{align*}
Then the lemma follows from the evaluation
	$$ \Zeta \left( \tfrac{1}{2}, W_0(s) \right) = \frac{1-q^{-1-2s}}{(1-q^{-\frac{1}{2}-s}) (1-q^{-\frac{1}{2}+s})} = \frac{1+q^{-\frac{1}{2}-s}}{1-q^{-\frac{1}{2}+s}},$$
	which is obtainable from \cite[Theorem 4.6.5]{Bu98}.
\end{proof}

\subsection{Case 2: Non-supercuspidal $\pi$ with $\fa(\pi)=2$}
\label{Case2NSSec}
	
We first consider non-supercuspidal $\pi_0$. Note that the case $\pi_0 \simeq \pi(\xi,\xi^{-1})$ was already included in \cite[\S 4]{BFW21+}. We recall the construction of the test function (with normalization)
\begin{align*}
C(g)=\frac{1}{\Vol(\gp{K}_0[\vp])} \phi_0(n(-\varpi^{-1})gn(\varpi)), \quad 
\phi_0 \Big(\begin{matrix} x_1 & x_2 \\ x_3 & x_4 \end{matrix}\Big) := \xi \Big( \frac{x_4}{x_1} \Big) \mathbbm{1}_{\gp{K}_0[\vp]} \Big(\begin{matrix} x_1 & x_2 \\ x_3 & x_4 \end{matrix}\Big).
\end{align*}
Note that if $\xi$ is quadratic, then the above function takes another form
	$$ C(g) = \frac{\xi(\det(g))}{\Vol(\gp{K}_0[\vp])} \mathbbm{1}_{\gp{K}_0[\vp]}(n(\varpi^{-1})gn(-\varpi^{-1})), $$
due to the transformation
	$$ \xi(x_1x_4-x_2x_3) = \xi(x_1x_4) = \xi \Big( \frac{x_4}{x_1} \Big), \quad \text{for } \Big(\begin{matrix} x_1 & x_2 \\ x_3 & x_4 \end{matrix}\Big)\in \gp{K}_0[\vp]. $$

For the quadratic character $\xi$, the assumption $\fa(\xi)=1$ is not essential. In fact, if $\fa(\xi)=n\in\{1,2\}$, replacing $\varpi^{-1}$ by $\varpi^{-n}$ and $\gp{K}_0[\vp]$ by $\gp{K}_0[\vp^n]$ gives the right test function. We shall give the relevant estimations in this slightly more general situation.

\begin{lemma}\label{LocMainBd2-1}
Suppose $C \in \cC_{\mathrm{c}}^{\infty}(\GL_2(\bF))$ is chosen as above and $n\in\{1,2\}.$
\begin{enumerate}[$(1)$]
\item For any irreducible admissible representation $\pi$ of $\PGL_2(\bF)$, the operator $\pi(C)$ is zero unless $\fa(\pi \otimes \xi) \leq n,$ in which case we have $\Mt_3(C \mid \pi) \geq 0$. Moreover, we have
	$$ \Mt_3(C \mid \mathrm{St}_{\xi}) \gg q^{-n}. $$
\item The dual weight $\Mt_4(C \mid \chi) $ vanishes unless $\fa(\chi) \leq n,$ in which case we have
	$$\Mt_4(C \mid \chi)\ll q^{-n}. $$
\end{enumerate}
\end{lemma}

\begin{proof}
All assertions are contained in \cite[Lemma 4.1 \& Corollary 4.8]{BFW21+}, except for the bound of $\Mt_3(C \mid \mathrm{St}_{\xi})$ which seems to be missing in the current proof of \cite[Lemma 4.1]{BFW21+}. We provide this simple verification as follows. Let $e_{\xi}$ be a new vector of $\mathrm{St}_{\xi}$ with Kirillov function $W_{\xi}(t) = \xi(t) \norm[t] \mathbbm{1}_{\vo}(t)$. Then we get, with the same argument as given in the proof of \cite[Lemma 4.1]{BFW21+}, the following bound
\begin{align*}
\Mt_3(C \mid \mathrm{St}_{\xi})
\geq\Big|\int_{\bF^{\times}} W_{\xi}(y) \psi(\varpi^{-n}y) \ud^{\times}y\Big|^2 \cdot \Big( \int_{\bF^{\times}} \extnorm{W_{\xi}(y)}^2 \ud^{\times}y \Big)^{-1} = q^{-n} (1+q^{-1}), 
\end{align*}
which is precisely of the desired form.
\end{proof}

\begin{lemma} \label{LocAuxBd2-1}
	Let $\xi$ be quadratic with $\fa(\xi)=1$, i.e., $\vp \nmid 2$. Then we have
\begin{align*}
	\Mt_4(C \mid \id, \tfrac{1}{2} + s)
	&= \zeta_{\vp}(1)(q+1)q^{-3}\sum_{\pm} \Bigg\{\frac{q^2\cdot q^{-(1\pm 2s)}}{( 1-q^{-(\frac{1}{2} \pm s)})^2}
	+\frac{2 \zeta_{\vp}(1) q\cdot q^{-(1\pm 2s)}}{1-q^{-(\frac{1}{2}\pm s)}}+\xi(\pm1)\Bigg\}.
\end{align*}
\end{lemma}

\begin{proof}
	This is a special case of the computation right before \cite[(5.1)]{BFW21+}.
\end{proof}

\subsection{Case 2: Supercuspidal $\pi$ with $\fa(\pi)=2$}
	
	Let $\bE$ be the unique unramified quadratic field extension of $\bF$ with ring of integers $\vo_{\bE}$ and prime ideal $\vp_{\bE}$. There is a character $\theta$ of $\bE^{\times}$ with conductor exponent $\fa(\theta)=1$, which is trivial on $\bF^{\times}$, such that $\pi_0 \simeq \pi(\theta)$, whose construction is recalled as follows (see \cite[\S 5.2.4]{FMP17}).
	
	Let $k_{\bE}$ and $k_{\bF}$ be the residue fields of $\bE$ and $\bF$, respectively, so that $k_{\bE} \simeq \mathbb{F}_{q^2}$ and $k_{\bF} \simeq \mathbb{F}_q$ as finite fields. Since $\vo_{\bE}^{\times}/\vo^{\times}(1+\vp_{\bE}) \simeq k_{\bE}^{\times} / k_{\bF}^{\times}$, the character $\theta$ is inflated from a character $\theta^{\flat}$ of $k_{\bE}^{\times} / k_{\bF}^{\times}$. We require $\theta^{\flat}$ to be regular, i.e., $\theta^{\flat} \neq (\theta^{\flat})^{\iota}$ for the action of the non-trivial element in $\mathrm{Gal}(k_{\bE}/k_{\bF})$ written by $x \mapsto x^{\iota}$. There is a cuspidal representation $\pi(\theta^{\flat})$ of $\GL_2(k_{\bF})$ of dimension $q-1$ constructed from and parametrized by $\theta^{\flat}$, whose character $\chi_{\theta^{\flat}}$ is given by \cite[(6.4.1)]{BuH06}
	\begin{equation}
	\chi_{\theta^{\flat}}(g) =
	\begin{cases}
	 q-1, & \text{if }g\sim \gp{Z}(k_{\bF}), \\
	 - 1, & \text{if } g\sim \gp{Z}(k_{\bF})\gp{N}(k_{\bF})-\gp{Z}(k_{\bF}), \\
	 - \theta^{\flat}(y) - \theta^{\flat}(y^{\iota}), & \text{if }g\sim k_{\bE}-k_{\bF}, \\
	 0, &\text{otherwise}.
	\end{cases}
\label{Dep0CharT}
\end{equation}
Here and in what follows, we adopt the convention $g\sim T$ to indicate that $g$ is conjugate with some element of the set $T$.

We regard $\pi(\theta^{\flat})$ as a representation of $\GL_2(\vo)$ by congruence modulo $\vp$, with extension to $\gp{Z}(\bF) \GL_2(\vo)$ by triviality on $\gp{Z}(\bF)$. Then $\pi(\theta)$ is the compact induction from $\pi(\theta^{\flat})$:
	$$ \pi(\theta) := \cInd_{\gp{Z}(\bF) \GL_2(\vo)}^{\GL_2(\bF)} \pi(\theta^{\flat}). $$
For $\kappa \in \GL_2(\vo)$, write $[\kappa]$ for its image in $\GL_2(k_{\bF})$ under the map modulo $\vp$. Define
	$$ C_{\theta}(g) := (q-1) \mathbbm{1}_{\GL_2(\vo)}(g) \chi_{\theta^{\flat}}([g]). $$

\begin{lemma}
For any irreducible admissible representation $\pi$ of $\PGL_2(\bF)$, the operator
$$\pi(\overline{C_{\theta}}) := \int_{\GL_2(\bF)} \overline{C_{\theta}(g)} \pi(g) \ud g$$
is zero unless $\pi \simeq \pi(\theta)$. Moreover, $\pi(\theta)(\overline{C_{\theta}})$ is the orthogonal projection on the $\pi(\theta^{\flat})$-isotypic part for the action of $\GL_2(\vo)$, which is isomorphic to $\pi(\theta^{\flat})$.
\label{Dep0Family}
\end{lemma}

\begin{proof}
By irreducibility and Frobenius reciprocity, the restriction of $\pi(\theta)$ to $\GL_2(\vo)$ contains $\pi(\theta^{\flat})$ with multiplicity one. Other assertions are then easy consequences of the Peter--Weyl theorem.
\end{proof}

The function $\vo \to \C^1, x \mapsto \psi(\varpi^{-1}x)$ can be viewed as a non-trivial character of $\vo/\vp = k_{\bF}$. All other non-trivial characters of $k_{\bF}$ are $x \mapsto \psi(\varpi^{-1}ux)$, which we denote by $\widetilde{\psi}_u(x)$, as $u$ runs over $\vo^{\times}/(1+\vp) = k_{\bF}^{\times}$. Observing the first two cases in (\ref{Dep0CharT}), we note
$$ \chi_{\theta^{\flat}}(n(x)) = -1 = \sum_{u \in k_{\bF}^{\times}}\widetilde{\psi}_u(x), \quad x \in k_{\bF}. $$
Consequently, there is an orthonormal basis $\{ e_u: u \in k_{\bF}^{\times}\}$ of $\pi(\theta^{\flat})$ characterized by
$$ \pi(\theta^{\flat})(n(x)).e_u = \widetilde{\psi}_u(x)e_u, \quad x \in k_{\bF}. $$
Moreover, we can choose $e_u$ so that
$$ e_u = \pi(\theta^{\flat})(a(u^{-1})).e_1, \quad a(u) = \Big(\begin{matrix} u & \\ & 1 \end{matrix}\Big). $$
We denote by $f_u \in \pi(\theta)$ the function with support in $\gp{Z}(\bF)\GL_2(\vo)$, and $f_u(\mathbbm{1}) = e_u$. Then the space generated by $\{f_u\}$ consists of all functions in $\pi(\theta)$ with support  in $\gp{Z}(\bF)\GL_2(\vo)$, i.e., the $\pi(\theta^{\flat})$-isotypic part in $\pi(\theta)$. If $K_u$ denotes the Kirillov function of $f_u$ with respect to the additive character $\psi$, then
$$ \Mt_3(\overline{C_{\theta}} \mid \pi(\theta)) = \sum_{u \in k_{\bF}^{\times}} |\Zeta(\tfrac{1}{2},K_u)|^2 \|K_u\|^{-2}. $$
	
\begin{lemma}
Let $\Pairing{\cdot}{\cdot}_{\flat}$ be the natural pairing between $\pi(\theta^{\flat})$ and its dual representation. Let $\{ e_u^{\vee}: u \in k_{\bF}^{\times}\}$ be the dual basis of $\{ e_u: u \in k_{\bF}^{\times}\}$. A non-trivial Whittaker functional on $\pi(\theta)$ is given by
$$ \ell(f) = \extPairing{\int_{\bF} f \Big( \Big(\begin{matrix} \varpi & \\ & 1 \end{matrix}\Big) \Big(\begin{matrix} 1 & x \\ & 1 \end{matrix}\Big)\Big) \psi(-x) \ud x }{e_1^{\vee}}_{\flat}. $$
\end{lemma}

\begin{corollary}
A possible choice of $K_u$ is given in the Kirillov model by
\begin{align}\label{eq:Whittakerfunctional}
K_u(y) = \mathbbm{1}_{\varpi^{-1} u (1+\vp)}(y), \quad y \in \bF^{\times}. 
\end{align}
Consequently, we get
$$ \Mt_3(\overline{C_{\theta}} \mid \pi(\theta)) = (q-1) \Vol(1+\vp,\ud^{\times}y) = 1. $$
\end{corollary}

\begin{proof}
	It suffices to prove the corollary in order to verify the non-vanishing of the formula defining the possible Whittaker functional \eqref{eq:Whittakerfunctional}. For $n \in \Z$ and $t \in \vo^{\times}$, we have
	$$ \ell(a(\varpi^n t).f_u) = \extPairing{\int_{\bF} f_u \Big( \Big(\begin{matrix} \varpi & \\ & 1 \end{matrix}\Big) \Big(\begin{matrix} 1 & x \\ & 1 \end{matrix}\Big)\Big(\begin{matrix} \varpi^n t & \\ & 1 \end{matrix}\Big)\Big) \psi(-x) \ud x }{e_1^{\vee}}_{\flat}. $$
	The integrand vanishes unless
	$$ \Big(\begin{matrix} \varpi & \\ & 1 \end{matrix}\Big) \Big(\begin{matrix} 1 & x \\ & 1 \end{matrix}\Big)\Big(\begin{matrix} \varpi^n t & \\ & 1 \end{matrix}\Big)\in \gp{Z}(\bF) \GL_2(\vo). $$
	Comparing the determinant and the lower right entry, the above condition is equivalent to 
	$$ n=-1, \quad x \in \vp^{-1}. $$
	It follows that
\begin{align*}
	\ell(a(\varpi^{-1} t).f_u) &= \extPairing{\int_{\vp^{-1}} f_u \Big( \Big(\begin{matrix} 1 & \varpi x \\ & 1 \end{matrix}\Big)\Big(\begin{matrix} t & \\ & 1 \end{matrix}\Big) \Big)\psi(-x) \ud x }{e_1^{\vee}}_{\flat} \\
	&= \extPairing{\int_{\vp^{-1}} \psi(-x) \pi(\theta^{\flat})\Big( \Big(\begin{matrix} 1 & \varpi x \\ & 1 \end{matrix}\Big)\Big(\begin{matrix} t & \\ & 1 \end{matrix}\Big) \Big).e_u \ud x }{e_1^{\vee}}_{\flat} \\
	&= \extPairing{\int_{\vp^{-1}} \psi(-x) \pi(\theta^{\flat})\Big( \Big(\begin{matrix} 1 & \varpi x \\ & 1 \end{matrix}\Big)\Big).e_{t^{-1}u} \ud x }{e_1^{\vee}}_{\flat} \\
	&= \extPairing{\int_{\vp^{-1}} \psi(-x) \psi(t^{-1}ux) \ud x \cdot e_{t^{-1}u}}{e_1^{\vee}}_{\flat} \\
	&= \mathbbm{1}_{u(1+\vp)}(t) \Vol(\vp^{-1}, \ud x).
\end{align*}
	We conclude the formula for $K_u$ by rescaling by the constant $\Vol(\vp^{-1}, \ud x)^{-1}$.
\end{proof}

We turn to the study of $\Mt_4(\overline{C_{\theta}} \mid \chi, \tfrac{1}{2})$. We have by definition
$$ \Mt_4(\overline{C_{\theta}} \mid \chi, \tfrac{1}{2})
= (q-1) \zeta_{\bF}(1)^4
\int_{\GL_2(\vo)}\overline{\chi_{\theta^{\flat}}\Big(\Big(\begin{matrix} x_1&x_2\\x_3&x_4\end{matrix}\Big)\Big)}
\chi\Big( \frac{x_1 x_4}{x_2 x_3} \Big) \frac{\ud x_1\ud x_2\ud x_3\ud x_4}{|x_1x_2x_3x_4|^{\frac{1}{2}}}. $$
Consider the following decomposition
$$ \GL_2(k_{\bF}) = G^{\emptyset} \sqcup G^{\{1\}} \sqcup G^{\{2\}} \sqcup G^{\{3\}} \sqcup G^{\{4\}} \sqcup G^{\{1,4\}} \sqcup G^{\{2,3\}}, $$
where for each subset $I\subseteq\{1,2,3,4\},$ $G^I$ defined by
	\begin{align*}
	G^I=\Big\{\Big(\begin{matrix} a_1 & a_2 \\ a_3 & a_4 \end{matrix}\Big)\in \GL_2(k_{\bF})
:a_j = 0\Longleftrightarrow j\in I\Big\}.
	\end{align*}
	Let $\tilde{G}^I \subset \GL_2(\vo)$ be the pre-image of $G^I$ under the map of modulo $\vp$. Introduce
	$$ \Mt_4^I := (q-1) \zeta_{\bF}(1)^4 \int_{\tilde{G}^I} \overline{\chi_{\theta^{\flat}} \Big(\Big( \begin{matrix} x_1 & x_2 \\ x_3 & x_4 \end{matrix} \Big)\Big)} \chi \Big( \frac{x_1 x_4}{x_2 x_3} \Big) \frac{\ud x_1\ud x_2\ud x_3\ud x_4}{|x_1x_2x_3x_4|^{\frac{1}{2}}}. $$
	We get a decomposition
	$$ \Mt_4(\overline{C_{\theta}} \mid \chi) = \Mt_4^{\emptyset} + \Mt_4^{\{1\}} + \Mt_4^{\{2\}} + \Mt_4^{\{3\}} + \Mt_4^{\{4\}} + \Mt_4^{\{1,4\}} + \Mt_4^{\{2,3\}} . $$
	
\begin{lemma}
	We have $\Mt_4^{\{1\}} = \Mt_4^{\{4\}} = 0$.
\end{lemma}
\begin{proof}
	By definition, we have
\begin{align*}
	\Mt_4^{\{1\}} &= (q-1) \zeta_{\bF}(1)^3 \int_{(\vo^{\times})^3} \overline{\chi_{\theta^{\flat}} \Big(\Big( \begin{matrix}  & x_2 \\ x_3 & x_4 \end{matrix} \Big)\Big)} \chi \Big( \frac{x_4}{x_2 x_3} \Big) \ud x_2 \ud x_3 \ud x_4 \cdot \int_{\vp} \chi(x_1) \norm[x_1]^{\frac{1}{2}} \ud^{\times}x_1 \\
	&= (q-1) \zeta_{\bF}(1)^3 \mathbbm{1}_{\fa(\chi)=0} \frac{\chi(\varpi) q^{-\frac{1}{2}}}{1-\chi(\varpi) q^{-\frac{1}{2}}} \cdot q^{-3} \left( S_1 + S_2 \right),
\end{align*}
	where $S_1$ resp. $S_2$ is defined below according as $\kappa\sim k_{\bE}-k_{\bF}$ resp. $\kappa\sim \gp{Z}(k_{\bF})\gp{N}(k_{\bF}) - \gp{Z}(k_{\bF})$ with
	$$ \kappa := \Big(\begin{matrix} 0 & x_2 \\ x_3 & x_4 \end{matrix}\Big). $$
	Note that the non-vanishing condition implies $\chi(x_4/x_2x_3)=1$ in the above integrand. Note also that $\kappa$ satisfies the equation $\kappa^2 - x_4 \kappa - x_2x_3 = 0$. If $\kappa\sim k_{\bE}-k_{\bF}$, then $\kappa\sim 2^{-1}x_4+\beta \sqrt{\varepsilon}$ for some $\beta \in k_{\bF}^{\times}$ such that $\beta^2 \varepsilon \equiv 4^{-1}x_4^2 + x_2x_3 \pmod{\vp}$. Hence
\begin{align*}
	S_1 &:= -\sum_{\beta \in k_{\bF}^{\times}/\{ \pm 1 \}} \sum_{\substack{a_2,a_3,a_4 \in k_{\bF}^{\times} \\ 4^{-1}a_4^2 + a_2a_3 = \beta^2 \varepsilon}}( \theta^{\flat}(2^{-1}a_4 + \beta \sqrt{\varepsilon}) + \theta^{\flat}(2^{-1}a_4 - \beta \sqrt{\varepsilon})) \\
	&= - (q-1)^2 \sum_{\beta \in k_{\bF}^{\times}} \theta^{\flat}(1+\beta \sqrt{\varepsilon}) = (q-1)^2.
\end{align*}
While the condition that $\kappa\sim\gp{Z}(k_{\bF})\gp{N}(k_{\bF}) - \gp{Z}(k_{\bF})$ is equivalent to $4^{-1}x_4^2 + x_2x_3 \in \vp$, we deduce
\begin{align*}
S_2 &:= \sum_{\substack{a_2,a_3,a_4 \in k_{\bF}^{\times} \\ 4^{-1}a_4^2 + a_2a_3 = 0}} (-1) = -(q-1)^2.
\end{align*}
Thus $S_1+S_2 = 0$, and $\Mt_4^{\{1\}}=0$. The treatment to $\Mt_4^{\{4\}} $ is quite similar and is omitted here.
\end{proof}

\begin{lemma}
	We have
	$$ \Mt_4^{\{2\}} = \Mt_4^{\{3\}} = - \mathbbm{1}_{\fa(\chi)=0} \frac{\chi^{-1}(\varpi)q^{-\frac{1}{2}}}{1-\chi^{-1}(\varpi)q^{-\frac{1}{2}}}. $$
\end{lemma}
\begin{proof}
	We only treat $\Mt_4^{\{3\}}$. By definition, we get
\begin{align*}
	\Mt_4^{\{3\}} &= (q-1) \zeta_{\bF}(1)^3  \int_{(\vo^{\times})^3} \overline{\chi_{\theta^{\flat}} \Big(\Big( \begin{matrix} x_1 & x_2 \\ 0 & x_4 \end{matrix} \Big)\Big)} \chi \Big( \frac{x_1 x_4}{x_2} \Big) \ud x_1 \ud x_2 \ud x_4 \cdot \int_{\vp} \chi^{-1}(x_3) |x_3|^{\frac{1}{2}} \ud^{\times}x_3 \\
	&= (q-1) \zeta_{\bF}(1)^3 \mathbbm{1}_{\fa(\chi)=0} \frac{\chi^{-1}(\varpi) q^{-\frac{1}{2}}}{1-\chi^{-1}(\varpi) q^{-\frac{1}{2}}} \cdot q^{-3} \sum_{a_1,a_2,a_4 \in k_{\bF}^{\times}} \overline{\chi_{\theta^{\flat}} \Big(\Big( \begin{matrix} a_1 & a_2 \\ 0 & a_4 \end{matrix} \Big)\Big)} \\
	&= (q-1) \zeta_{\bF}(1)^3 \mathbbm{1}_{\fa(\chi)=0} \frac{\chi^{-1}(\varpi) q^{-\frac{1}{2}}}{1-\chi^{-1}(\varpi) q^{-\frac{1}{2}}} \cdot q^{-3} (q-1)^2 (-1)
\end{align*}
	as desired.
\end{proof}

\begin{lemma}
	We have
	$$ \Mt_4^{\{2,3\}} = (q-1) \mathbbm{1}_{\fa(\chi)=0} \cdot \Big( \frac{\chi^{-1}(\varpi) q^{-\frac{1}{2}}}{1 - \chi^{-1}(\varpi) q^{-\frac{1}{2}}} \Big)^2. $$
\end{lemma}
\begin{proof}
	By definition, we have
\begin{align*}
	\Mt_4^{\{2,3\}} &= (q-1) \zeta_{\bF}(1)^2  \int_{(\vo^{\times})^2} \overline{\chi_{\theta^{\flat}} \Big(\Big( \begin{matrix} x_1 & 0 \\  0& x_4 \end{matrix} \Big)\Big)} \chi \left( x_1 x_4 \right) \ud x_1 \ud x_4 \cdot \Big(\int_{\vp} \chi^{-1}(x_2) |x_2|^{\frac{1}{2}} \ud^{\times}x_2 \Big)^2\\
	&= (q-1) \zeta_{\bF}(1)^2 \id_{\fa(\chi)=0} \cdot \Big( \frac{\chi^{-1}(\varpi) q^{-\frac{1}{2}}}{1 - \chi^{-1}(\varpi) q^{-\frac{1}{2}}} \Big)^2 \cdot q^{-2} \sum_{a_1,a_4 \in k_{\bF}^{\times}} \overline{\chi_{\theta^{\flat}} \Big(\Big( \begin{matrix} a_1 & 0 \\ 0 & a_4 \end{matrix} \Big)\Big)}.
\end{align*}
The lemma follows by noting that the inner sum is equal to $(q-1)^2$.
\end{proof}

\begin{lemma}
	We have 
	$$ \Mt_4^{\{1,4\}} = - (q-1) \id_{\fa(\chi)=0} \cdot \Big( \frac{\chi(\varpi) q^{-\frac{1}{2}}}{1 - \chi(\varpi) q^{-\frac{1}{2}}} \Big)^2 \cdot \overline{\theta^{\flat}(\sqrt{\varepsilon})}. $$
\end{lemma}

\begin{proof}
	By definition, we have
\begin{align*}
	\Mt_4^{\{1,4\}}&= (q-1) \zeta_{\bF}(1)^2 \int_{(\vo^{\times})^2} \overline{\chi_{\theta^{\flat}} \Big(\Big( \begin{matrix}  0&  x_2\\  x_3 &  0\end{matrix} \Big)\Big)}  \chi^{-1} \left( x_2 x_3 \right) \ud x_2 \ud x_3 \cdot \Big(\int_{\vp} \chi(x_1) |x_1|^{\frac{1}{2}} \ud^{\times}x_1\Big)^2\\
	&= (q-1) \zeta_{\bF}(1)^2 \id_{\fa(\chi)=0} \cdot \Big( \frac{\chi(\varpi) q^{-\frac{1}{2}}}{1 - \chi(\varpi) q^{-\frac{1}{2}}} \Big)^2 \cdot q^{-2} \sum_{a_2,a_3 \in k_{\bF}^{\times}} \overline{\chi_{\theta^{\flat}} \Big(\Big( \begin{matrix}  0& a_2 \\ a_3 &  0\end{matrix}\Big)\Big)}.
\end{align*}
	The matrix $(\begin{smallmatrix}  0& a_2 \\ a_3 &  0\end{smallmatrix})$ being of trace $0$, can not be conjugate to an element in $\gp{Z}(k_{\bF}) \gp{N}(k_{\bF})$. It is conjugate to $\beta \sqrt{\varepsilon}$ for some $\beta \in k_{\bF}^{\times}$ if and only if $a_2a_3 = \beta^2 \varepsilon$. Thus
	$$ \sum_{a_2,a_3 \in k_{\bF}^{\times}} \chi_{\theta^{\flat}} \Big(\Big( \begin{matrix}  0& a_2 \\ a_3 &  0\end{matrix}\Big)\Big)= -(q-1) \sum_{\beta \in k_{\bF}^{\times}/\{ \pm 1 \}} (\theta^{\flat}(\beta \sqrt{\varepsilon})+\theta^{\flat}(-\beta \sqrt{\varepsilon}) ) = - (q-1)^2\theta^{\flat}(\sqrt{\varepsilon}).
	$$
	The lemma follows readily.
\end{proof}

	By definition, we have
	$$ \Mt_4^{\emptyset} = (q-1) \zeta_{\bF}(1)^4 \int_{(\vo^{\times})^4}\overline{\chi_{\theta^{\flat}} \Big(\Big( \begin{matrix} x_1 & x_2 \\ x_3 & x_4 \end{matrix} \Big)\Big)} \chi \Big( \frac{x_1 x_4}{x_2 x_3} \Big)\id_{x_1x_4-x_2x_3 \in\vo^{\times}}\ud x_1\ud x_2\ud x_3\ud x_4.$$
	It is easy to see that $\Mt_4^{\emptyset} $ vanishes unless $\fa(\chi) \leq 1$, since, for example, $\overline{\chi_{\theta^{\flat}}(\cdot)}$ is invariant under the change of variable $x_1 \mapsto x_1(1+u)$ for any $u \in \vp$. In this case, we may regard $\chi$ as a character of $k_{\bF}^{\times}$, and rewrite the above expression as
	$$ \Mt_4^{\emptyset} = (q-1)^{-3} \sum_{\substack{a_1,a_2,a_3,a_4\in k_{\bF}^{\times} \\ a_1a_4 - a_2a_3 \neq 0}}\overline{\chi_{\theta^{\flat}} \Big(\Big( \begin{matrix} a_1 & a_2 \\ a_3 & a_4 \end{matrix} \Big)\Big)} \chi \Big( \frac{a_1 a_4}{a_2 a_3} \Big). $$
	The conjugacy class of the matrix $y :=(\begin{smallmatrix} a_1 & a_2 \\ a_3 & a_4 \end{smallmatrix})$ is intimately related to $\Delta := (a_1-a_4)^2+4a_2a_3$. Precisely, we have
	\begin{align*}
	y \sim \gp{Z}(k_{\bF}) \gp{N}(k_{\bF}) - \gp{Z}(k_{\bF}) & \Longleftrightarrow  \Delta = 0
	\end{align*}
	and
	\begin{align*}
	y \sim k_{\bE}-k_{\bF} & \Longleftrightarrow  \Delta \in (k_{\bF}^{\times})^2 \varepsilon.
	\end{align*}

	Moreover, $\chi_{\theta^{\flat}}(y)$ vanishes if $\Delta \in (k_{\bF}^{\times})^2$. We can thus decompose
	$$ \Mt_4^{\emptyset} = \Mt_4^{\emptyset,0} + \Mt_4^{\emptyset,1} $$
	with
	\begin{align*}
	\Mt_4^{\emptyset,0}
	&= -(q-1)^{-3} \sum_{\substack{a_1,a_2,a_3,a_4\in k_{\bF}^{\times} \\ a_1a_4 - a_2a_3 \neq 0 \\ \Delta = 0}} \chi \Big( \frac{a_1a_4}{a_2a_3} \Big), \\
	\Mt_4^{\emptyset,1}
	&= -(q-1)^{-3} \sum_{\beta \in k_{\bF}^{\times}} \sum_{\substack{a_1,a_2,a_3,a_4\in k_{\bF}^{\times} \\ a_1a_4 - a_2a_3 \neq 0 \\ \Delta = 4\beta^2 \varepsilon}} \overline{\theta^{\flat} \Big(\frac{a_1+a_4}{2}+\beta \sqrt{\varepsilon} \Big)} \chi \Big( \frac{a_1a_4}{a_2a_3} \Big). 
	\end{align*}
	
Trivially we have $\Mt_4^{\emptyset,0}\ll1.$ To analyze $\Mt_4^{\emptyset,1}$, we re-parametrize the sum with $\alpha \in k_{\bF}, \beta \in k_{\bF}^{\times}, t \in k_{\bF}-\{ 0,1 \}$ such that
$$ a_1+a_4 = 2\alpha, \quad \Delta = (a_1-a_4)^2+4a_2a_3 = 4\beta^2 \varepsilon, \quad a_1a_4 = t a_2a_3. $$
Equivalently, the change of variables is
$$ a_1+a_4 = 2\alpha, \quad a_1a_4 = (1-t^{-1})^{-1} (\alpha^2 - \beta^2 \varepsilon), \quad a_1a_4 = t a_2a_3. $$
We can rewrite
\begin{align*} 
\Mt_4^{\emptyset,1}
&= -(q-1)^{-2} \sum_{\substack{\alpha,\beta,t,a_1,a_4\in k_{\bF}\\ \beta\neq0,~t\neq 0,1\\a_1+a_4 =2\alpha,~a_1a_4 (1-t^{-1})=(\alpha^2 - \beta^2 \varepsilon)}}\overline{\theta^{\flat}(\alpha+\beta \sqrt{\varepsilon})} \chi(t)\\
&= -(q-1)^{-2} \sum_{\substack{\alpha,\beta,t,a_1,a_4\in k_{\bF}\\ \beta\neq0,~t\neq 0,1\\a_1+a_4 =2\alpha,~a_1a_4 (1-t^{-1})=(\alpha^2 - \beta^2 \varepsilon)}}\overline{\theta^{\flat}(\alpha/\beta+ \sqrt{\varepsilon})} \chi(t).
\end{align*}
Making the change of variable $\alpha\rightarrow \alpha\beta$ for given $\beta\in k_\bF^\times,$ we further write
\begin{align*} 
\Mt_4^{\emptyset,1}
&= -(q-1)^{-2} \sum_{\substack{\alpha,\beta,t,a_1,a_4\in k_{\bF}\\ \beta\neq0,~t\neq 0,1\\a_1+a_4 =2\alpha\beta,~a_1a_4 (1-t)= \beta^2(\alpha^2 -\varepsilon)}}\overline{\theta^{\flat}(\alpha+\sqrt{\varepsilon})} \overline{\chi}(t)\\
&= -(q-1)^{-2} \sum_{\substack{\alpha,\beta,t,a_1,a_4\in k_{\bF}\\a_1+a_4 =2\alpha\beta,~a_1a_4 (1-t)= \beta^2(\alpha^2 -\varepsilon)}}\overline{\theta^{\flat}(\alpha+\sqrt{\varepsilon})} \overline{\chi}(t)+O(1)\\
&= -(q-1)^{-2} \sum_{\substack{\alpha,t,a_1,a_4\in k_{\bF}\\ 4\alpha^2 a_1a_4 (1-t)=(a_1+a_4)^2(\alpha^2-\varepsilon)}}\overline{\theta^{\flat}(\alpha+\sqrt{\varepsilon})} \overline{\chi}(t)+O(1).
\end{align*}
For given $\alpha,t\in k_{\bF}$, the number of tuples $(a_1,a_4)\in k_\bF^2$ satisfying $4\alpha^2 a_1a_4 (1-t)=(a_1+a_4)^2(\alpha^2-\varepsilon)$ is exactly $q(1+\phi(\alpha^2t-\varepsilon)\phi(t-1)),$
where $\phi$ is the unique non-trivial quadratic character of $k_{\bF}^{\times}$. It then follows that
\begin{align*} 
\Mt_4^{\emptyset,1}
&= -q(q-1)^{-2} \sum_{\alpha,t\in k_{\bF}}\overline{\theta^{\flat}(\alpha+\sqrt{\varepsilon})} \overline{\chi}(t)(1+\phi(\alpha^2t-\varepsilon)\phi(t-1))+O(1)\\
&= -q(q-1)^{-2} \sum_{\alpha,t\in k_{\bF}}\overline{\theta^{\flat}(\alpha+\sqrt{\varepsilon})}\chi(t)(1+\phi(\alpha^2-\varepsilon t)\phi(1-t))+O(1).
\end{align*}
If $\fa(\chi)=0$, then $\chi(t)=1$ for all $t\in k_\bF^\times,$ so that
\begin{align*} 
\Mt_4^{\emptyset,1}
&= -q(q-1)^{-1} \sum_{\alpha\in k_{\bF}}\overline{\theta^{\flat}(\alpha+\sqrt{\varepsilon})}-q(q-1)^{-2} \sum_{\alpha\in k_{\bF}}\overline{\theta^{\flat}(\alpha+\sqrt{\varepsilon})}\sum_{t\in k_{\bF}^\times}\phi(\alpha^2-\varepsilon t)\phi(1-t)+O(1)\\
&=-q(q-1)^{-2} \sum_{\alpha\in k_{\bF}}\overline{\theta^{\flat}(\alpha+\sqrt{\varepsilon})}\sum_{t\in k_{\bF}}\phi(\alpha^2-\varepsilon t)\phi(1-t)+O(1)\\
&\ll 1.
\end{align*}
If $\fa(\chi)=1$, we have
\begin{align*} 
\Mt_4^{\emptyset,1}
&= -q(q-1)^{-2} S(\theta^{\flat},\chi)+O(1),
\end{align*}
where
\begin{align}\label{eq:charactersum-1st}
S(\theta^{\flat},\chi)&=\sum_{\alpha,t\in k_{\bF}}\overline{\theta^{\flat}(\alpha+\sqrt{\varepsilon})}\chi(t)\phi(\alpha^2-\varepsilon t)\phi(1-t).
\end{align}

\begin{proposition}\label{prop:charactersum-1st}
For each non-trivial character $\chi$ of $k_{\bF}^{\times},$ we have
$$ S(\theta^{\flat}, \chi) \ll q. $$
\end{proposition}

It is trivial that $|S(\theta^{\flat}, \chi)|\leqslant q^2$, thus Proposition \ref{prop:charactersum-1st} captures double squareroot cancellations. The proof of Proposition \ref{prop:charactersum-1st} will be given in Section \ref{sec:charactersums}
(re-stated in Proposition \ref{prop:charactersum}). More precisely, we will associate $ S(\theta^{\flat}, \chi) $ with hypergeometric sums of Katz, so that $\ell$-adic cohomology enters in the picture.

Admitting Proposition \ref{prop:charactersum-1st}, we easily deduce the following result by summarizing the results obtained in previous paragraphs.

\begin{lemma}\label{LocMainBd2-2}
The dual weight $\Mt_4(\overline{C_{\theta}} \mid \chi, \frac{1}{2})$ vanishes unless $\fa(\chi) \leq 1,$ in which case we have 
$$\Mt_4(\overline{C_{\theta}} \mid \chi, \tfrac{1}{2})\ll 1. $$
\end{lemma}

\begin{lemma}\label{LocAuxBd2-2}
We have
$$ \Mt_4 ( \overline{C_{\theta}} \mid \id, \tfrac{1}{2} + s)
= -\frac{2q^{s-\frac{1}{2}}}{1-q^{s-\frac{1}{2}}} + (q-1) \Big( \frac{q^{s-\frac{1}{2}}}{1-q^{s-\frac{1}{2}}} \Big)^2 - \overline{\theta^{\flat}(\sqrt{\varepsilon})} (q-1) \Big( \frac{q^{-s-\frac{1}{2}}}{1-q^{-s-\frac{1}{2}}} \Big)^2 +\frac{2}{q-1}. $$
\end{lemma}

	\subsection{Case 3: $\fa(\pi)=3$}
	\label{sec:Case3}

	According to \cite{KL15}, $\pi_0 \in \{ \sigma_t^1, \sigma_t^{-1} \}$ belongs to a family, parametrized by $t \in (\vo/\vp)^{\times}$, consisting of two supercuspidal representations (and an implicitly chosen additive character of level $1$). Each $\sigma_t^{\zeta}$ contains a unique (up to scalar) vector $f_t^{\zeta}$, which transforms as a character $\chi_t$ of the compact open subgroup $\gp{K}_1[\vp]$. We recall
$$
\gp{K}_1[\vp] := \Big\{ \Big(\begin{matrix} x_1 & r_1 \\ \varpi r_2 & x_2 \end{matrix}\Big) \in \GL_2(\vo) : x_1,x_2 \in 1+\vp, ~ r_2 \in \vo \Big\},
$$
$$
\chi_t \Big(\Big(\begin{matrix} x_1 & r_1 \\ \varpi r_2 & x_2 \end{matrix}\Big)\Big) = \psi(\varpi^{-1}(r_1+tr_2)),
$$
where $\psi$ is the standard additive character \`a la Tate: trivial on $\vo$ but not on $\vp^{-1}$. Define
$$ C_t(g) := \frac{\mathbbm{1}_{\gp{K}_1[\vp]}(g) \chi_t(g)}{\Vol(\gp{K}_1[\vp])}. $$

\begin{lemma}
For any irreducible admissible representation $\pi$ of $\PGL_2(\bF)$, the operator
$$ \pi(\overline{C_t}) := \int_{\GL_2(\bF)} \overline{C_t(g)} \pi(g) dg$$
is zero unless $\pi \simeq \sigma_t^{\zeta}$ for some $\zeta \in \{ \pm 1 \}$. Moreover, $\sigma_t^{\zeta}(\overline{C_t})$ is the orthogonal projection on the line containing $f_t^{\zeta}$.
\label{CubeFamily}
\end{lemma}

\begin{proof}
	The range of $\pi(\overline{C_t})$ is the subspace of vectors transforming as $\chi_t$ on $\gp{K}_1[\vp]$. If it is non-zero, then, by Frobenius reciprocity, $\pi$ must be a subrepresentation of
	$$ \pi_t := \cInd_{\gp{Z} \gp{K}_1[\vp]}^{\GL_2(\bF)} \widetilde{\chi_t}, $$
	where $\widetilde{\chi_t}$ is the extension by triviality on $\gp{Z}$ of $\chi_t$. By \cite[Theorem 4.4]{KL15}, we have $\pi_t \simeq \sigma_t^1 \oplus \sigma_t^{-1}$. Thus the first assertion follows. The ``moreover'' part is an immediate consequence.
\end{proof}

	We proceed to the study of $\Mt_3(\overline{C_t} \mid \pi)$. By Lemma \ref{CubeFamily}, it is non-zero only if $\pi \simeq \sigma_t^{\zeta}$ for some $\zeta \in \{ \pm 1 \}$. If $W_t^{\zeta}$ denotes the Whittaker function of $f_t^{\zeta}$ with respect to $\psi$, then
	$$ \Mt_3(\overline{C_t} \mid \sigma_t^{\zeta}) = |\Zeta(\tfrac{1}{2}, W_t^{\zeta})|^2 \|W_t^{\zeta}\|^{-2}, $$
	where the norm is computed in the Kirillov model.
	
\begin{lemma} \label{lem:WhitFormula}
	We have $W_t^{\zeta}(a(y)) = \mathbbm{1}_{\varpi^{-1}(1+\vp)}(y)$.
\end{lemma}

\begin{proof}
	See \cite[Lemma A.2.1]{LPW23}.
\end{proof}
\begin{corollary}
	We have $\Mt_3(\overline{C_t} \mid \sigma_t^{\zeta}) = \Vol(1+\vp, \ud^{\times}y) = (q-1)^{-1}$.
\end{corollary}
\begin{proof}
	By Lemma \ref{lem:WhitFormula}, we get
	$$\Zeta(\tfrac{1}{2}, W_t^{\zeta})= \int_{\bF^{\times}} W_t^{\zeta}(a(y)) \ud^{\times}y = \Vol(1+\vp,\ud^{\times}y), $$
	$$\|W_t^{\zeta}\|^2 = \int_{\bF^{\times}} |W_t^{\zeta}(a(y))|^2 \ud^{\times}y = \Vol(1+\vp,\ud^{\times}y). $$
	The desired formula follows readily.
\end{proof}

	We turn to the study of $\Mt_4(\overline{C_t} \mid \chi)$, which by definition is
\begin{align*}
\Mt_4(\overline{C_t} \mid \chi)
&=\int_{(\bF^{\times})^4}\overline{C_t}\Big(\begin{matrix} x_1 & x_2 \\ x_3 & x_4 \end{matrix}\Big)\chi \Big( \frac{x_1 x_4}{x_2 x_3} \Big) |x_1x_2x_3x_4|^{\frac{1}{2}}\ud^{\times}x_1\ud^{\times}x_2\ud^{\times}x_3\ud^{\times}x_4\\
&=\Vol(\gp{K}_1[\vp])^{-1}\cdot\Big(\int_{1+\vp} \chi(x_1) \ud^{\times}x_1\Big)^2\cdot \Big(\int_{\vo} \psi(-\varpi^{-1} x_2) \chi^{-1}(x_2) \norm[x_2]^{\frac{1}{2}} \ud^{\times}x_2\Big)\\
&\ \ \ \ \cdot \Big(\int_{\vp} \psi(-\varpi^{-2} t x_3) \chi^{-1}(x_3) \norm[x_3]^{\frac{1}{2}} \ud^{\times}x_3\Big).
\end{align*}
The first two integrals vanish unless $\fa(\chi) \leq 1$, in which case we have
\begin{multline*} 
	\Mt_4(\overline{C_t} \mid \chi) = \frac{\Vol(1+\vp,\ud^{\times}y)^2}{\Vol(\gp{K}_1[\vp])} \int_{\vo} \psi(-\varpi^{-1} x) \chi^{-1}(x) \norm[x]^{\frac{1}{2}} \ud^{\times}x \cdot \chi^{-1}(\varpi)q^{-\frac{1}{2}} \int_{\vo} \psi(-\varpi^{-1} t x) \chi^{-1}(x) \norm[x]^{\frac{1}{2}} \ud^{\times}x. 
\end{multline*}
	
Each of the last integrals defines a Gauss sum, for which we appeal to the following estimate.
\begin{lemma}\label{lm:Gausssum-vo}
For any additive character $\psi_1$ of level $1$ and any unitary $\chi$ with $\fa(\chi) \leq 1$, we have
$$ \int_{\vo} \psi_1(x) \chi(x) \norm[x]^{\frac{1}{2}} \ud^{\times}x \ll q^{-\frac{1}{2}}. $$
\end{lemma}

\begin{proof}
If $\fa(\chi)=0$, then we have
\begin{align*} 
\int_{\vo} \psi_1(x) \chi(x) \norm[x]^{\frac{1}{2}} \ud^{\times}x &= \int_{\vo-\vp} \psi_1(x) d^{\times}x + \int_{\vp} \chi(x) \norm[x]^{\frac{1}{2}} \ud^{\times}x= -\frac{1}{q-1} + \frac{\chi(\varpi) q^{-\frac{1}{2}}}{1 - \chi(\varpi) q^{-\frac{1}{2}}} \ll q^{-\frac{1}{2}}.
\end{align*}
If $\fa(\chi)=1$, then we have by \cite[Proposition 4.6]{Wu14}
\begin{align*} 
\int_{\vo} \psi_1(x) \chi(x) \norm[x]^{\frac{1}{2}} \ud^{\times}x = \int_{\vo^{\times}} \psi_1(x) \chi(x) \ud^{\times}x \ll q^{-\frac{1}{2}}.
\end{align*}
This completes the proof of the lemma.
\end{proof}

\begin{corollary}\label{LocMainBd3}
The dual weight $\Mt_4(\overline{C_t} \mid \chi, \frac{1}{2}) $ vanishes unless $\fa(\chi) \leq 1,$ in which case we have $$\Mt_4(\overline{C_t} \mid \chi, \tfrac{1}{2}) \ll q^{-\frac{1}{2}}.$$
\end{corollary}

\begin{lemma}\label{LocAuxBd3}
We have
$$ \Mt_4( \overline{C_t} \mid \id, \tfrac{1}{2} + s) = (q+1)q^{s-\frac{1}{2}} \Big( -\frac{1}{q-1} + \frac{q^{s-\frac{1}{2}}}{1-q^{s-\frac{1}{2}}} \Big)^2. $$
\end{lemma}

\begin{proof}
	This follows easily from the above computation in the case $\fa(\chi)=0$.
\end{proof}

\section{Bound for Double Character Sums}
\label{sec:charactersums}

\subsection{Statement}
Let $\F_q$ be a finite field with $q$ elements of characteristic $p$. Let $\omega$ be a primitive element in $\F_{q^2}$ such that $\omega^2\in\F_q.$ For characters $\chi,\eta$ of $\F_q^\times$ and $\rho$ of $\F_{q^2}^\times,$ we define the double character sum
\begin{align}\label{eq:charactersum}
S(\chi,\eta;\rho)= \sum_{\alpha\in\F_q} \rho(\alpha+\omega)\sum_{t\in\F_q} \chi(t)\eta(\alpha^2-\omega^2 t)\overline{\eta}(1-t). 
\end{align}
We reduce $S(\chi,\eta;\rho)$ to the original sum \eqref{eq:charactersum-1st} if $\eta$ is quadratic.
Proposition \ref{prop:charactersum-1st} is a special case of the following general bound, for which $\eta$ can be taken as an arbitrary non-trivial character of $\F_q^\times.$

\begin{proposition}\label{prop:charactersum} 
Let $\chi,\eta$ be non-trivial characters of $\F_q^\times,$ and let $\rho$ be a non-trivial character of $\F_{q^2}^\times.$ Then we have
\begin{align*}
|S(\chi,\eta;\rho)|\leqslant1000 q.
\end{align*}
\end{proposition}

We prove Proposition \ref{prop:charactersum} linearly. After recalling Jacobi sums and Gauss sums in finite fields, we will introduce $\ell$-adic sheaves and trace functions together with hypergeometric sums of Katz with certain geometric features. Based on such preliminaries, the proof splits into two parts: associating $S(\chi,\eta;\rho)$ with hypergeometric sums, applying Deligne's work on Riemann Hypothesis over finite fields along with Katz's observations on Lang torsors.

\subsection{Jacobi sum and Gauss sum}
\label{sec:JacobiGauss}
Denote by $\fG$ the character group of $\F_q^\times.$ The trivial character in $\fG$ is denoted by $\chi_0.$ 
For two characters $\chi_1,\chi_2\in \fG,$ define the Jacobi sum
\begin{align*}
J(\chi_1,\chi_2):= \sum_{\alpha\in\F_q}\chi_1(\alpha)\chi_2(1-\alpha).
\end{align*}
We adopt the convention that $\chi(0)=0$ for each $\chi\in \fG.$ The Jacobi sum is intimately connected with the Gauss sum
\begin{align*}
\tau(\chi,\psi):= \sum_{\alpha\in\F_q}\chi(\alpha)\psi(\alpha),
\end{align*}
where $\psi$ is an additive character of $\F_q.$ We say an additive character $\psi$ is canonical if for all $x\in\F_q,$
\begin{align*}
\psi(x)=\ue^{2\pi i\Tr(x)/p},
\end{align*}
where $\Tr:\F_q\rightarrow\F_p$ is the trace map. If $\psi$ is canonical, we write $\tau(\chi,\psi)=\tau(\chi).$
It is an easy excise to show that
\begin{align*}
\tau(\chi,\psi)=-1
\end{align*}
if $\chi$ is trivial and $\psi$ is non-trivial.

The following lemma associates Jacobi sums with Gauss sums, which is well-known, and should exist in literature for quite a long time; see Lidl and Niederreiter \cite[Theorem 5.21]{LN97} for instance.

\begin{lemma}\label{lm:Jacobi-Gauss}
Let $\chi_1,\chi_2\in \fG.$ Then $J(\chi_1,\chi_2)=q-2$ if $\chi_1,\chi_2$ are both trivial, and otherwise
\begin{align}\label{eq:Jacobi-Gauss}
J(\chi_1,\chi_2)= q^{-1}\tau(\chi_1,\psi)\tau(\chi_2,\psi)\overline{\tau(\chi_1\chi_2,\psi)}
\end{align}
for each non-trivial additive character $\psi$ of $\F_q.$
\end{lemma}

\begin{remark}
It seems more common to write
\begin{align*}
J(\chi_1,\chi_2)= \frac{\tau(\chi_1,\psi)\tau(\chi_2,\psi)}{\tau(\chi_1\chi_2,\psi)}
\end{align*}
if $\chi_1,\chi_2$ and $\chi_1\chi_2$ are all non-trivial, and the remaining cases can be 
verified manually.
\end{remark}

\subsection{$\ell$-adic sheaves, trace functions and hypergeometric sums}
\label{sec:sheaftracefunctionhypergeometric}
In this subsection, we introduce the terminology on trace functions of $\ell$-adic sheaves on $\mathbb{P}_{\F_q}^1$ 
following the manner of Fouvry, Kowalski and Michel \cite{FKM15a, FKM15b}.

\subsubsection{Trace functions}
Let $\ell\neq  p$ be an auxiliary prime, and fix an isomorphism $\iota:\overline{\Q}_\ell\rightarrow\C$. The functions $K(x)$ modulo $p$ that we
consider are the trace functions of suitable constructible sheaves on $\A^1_{\F_q}$
evaluated at $x\in\F_q$. To
be precise, we will consider middle-extension sheaves on $\mathbb{P}^1_{\F_q}$ 
and we refer to the following definition after Katz \cite[Section 7.3.7]{Ka88}.

\begin{definition}[Trace functions]\label{def:tracefunction}
Let $\cF$ be an $\ell$-adic middle-extension sheaf pure of weight zero, 
which is lisse on an open set $U$. The trace function associated to $\cF$ is defined by
\begin{align*}
K:x\in\F_q\mapsto\iota(\tr(\bFrob_x\mid V_\cF)),
\end{align*}
where $\bFrob_x$ denotes the geometric Frobenius at $x\in\F_q,$ and $V_\cF$ is a finite dimensional $\overline{\Q}_\ell$-vector space, which is corresponding to a continuous finite-dimensional Galois representation and unramified at every closed point $x$ of $U.$
\end{definition}

We need an invariant to measure the geometric complexity of a trace function, which can be given by some numerical invariants of the underlying sheaf.
\begin{definition}[Conductor] \label{def:conductor} 
For an $\ell$-adic middle-extension sheaf $\cF$ on $\mathbb{P}^1_{\F_q}$ of rank $\rank(\cF)$,
we define the $($analytic$)$ conductor of $\cF$ to be
\begin{align*}  
\fc(\cF) := \rank(\cF) + |S(\cF)|+\sum_{x\in S(\cF)} \Swan_x(\cF),
\end{align*}
where $S(\cF)\subset\mathbb{P}^1(\overline{\F}_q)$ denotes the $($finite$)$ set of singularities of $\cF,$ 
and $\Swan_x(\cF)$ $(\geqslant 0)$ denotes the Swan conductor of $\cF$ at $x$ $($see {\rm \cite{Ka80}}$).$
\end{definition}

There are fruitful examples of trace functions arising in analytic number theory, among which we would like to mention additive and multiplicative characters modulo $p$, as well as Kloosterman sums and general hyper-Kloosterman sums. In what follows, we would like to introduce hypergeometric sums, generalizing the so-called hyper-Kloosterman sums. The conductor defined in Definition \ref{def:conductor} is very crucial in applications to analytic number theory: in many situations we need to show the conductors of underlying sheaves remain bounded as the size of finite field grows.

	\subsubsection{Hypergeometric sums}

	We now consider hypergeometric sums introduced by Katz (see \cite[Chapter 8]{Ka90}). Let $m,n$ be two non-negative integers, and suppse $\boldsymbol{\chi}=(\chi_i)_{1\leqslant i\leqslant m}$ and $\boldsymbol{\eta}=(\eta_j)_{1\leqslant j\leqslant n}$ are two tuples of characters in $\fG$, and $\psi$ is the canonical additive character of $\bF_q$. Katz introduced the following hypergeometric sum
\begin{align*}
H(t,q;\boldsymbol\chi,\boldsymbol\eta)
:=\frac{(-1)^{m+n-1}}{q^{(m+n-1)/2}}\mathop{\sum\sum}_{\substack{\bx\in(\F_q^\times)^m,\by\in(\F_q^\times)^n\\ \fn(\bx)=t\fn(\by)}}\boldsymbol\chi(\bx)
\overline{\boldsymbol\eta(\by)}\psi(\ft(\bx)-\ft(\by))
\end{align*}
for $t\in\F_q^\times,$ where, for $\bx=(x_1,x_2,\cdots,x_m)\in(\F_q^\times)^m$, 
\begin{align*}
\boldsymbol\chi(\bx)=\prod_{1\leqslant i\leqslant m}\chi_i(x_i),
\end{align*}
\begin{align}\label{eq:T(x)N(x)}
\ft(\bx)=x_1+x_2+\cdots+x_m,\ \ \fn(\bx)=x_1x_2\cdots x_m,
\end{align}
and the notation with $\by$ can be defined in the same way. We say $\boldsymbol{\chi}$ and $\boldsymbol{\eta}$ are {\it disjoint} if $\chi_i\neq\eta_j$ for all $1\leqslant i\leqslant m$ and $1\leqslant j\leqslant n.$

In general, Katz \cite{Ka90} performed a very systematic study on geometric features of $H(t,q;\boldsymbol\chi,\boldsymbol\eta)$ and the underlying sheaf. We now summarize some of them, which turn out to be very crucial in our study on the double character sum \eqref{eq:charactersum}.

\begin{lemma}\label{lm:hypergeometric-weightranklisse}
With the above notation,
if $\boldsymbol\chi$ and $\boldsymbol\eta$ are disjoint, then for any $\ell\neq p,$ there exists a geometrically irreducible $\ell$-adic middle-extension sheaf $\cH(\boldsymbol\chi,\boldsymbol\eta)$ on $\A_{\F_q}^1$ with trace function given by $t\mapsto H(t,q;\boldsymbol\chi,\boldsymbol\eta),$ such that it is
\begin{itemize}
\item pointwise pure of weight zero and of rank $\max\{m,n\};$
\item lisse on $\mathbf{G}_{m,\F_q}$, if $m\neq n;$
\item lisse on $\mathbf{G}_{m,\F_q}-\{1\},$ if $m=n.$
\end{itemize}
\end{lemma}

Lemma \ref{lm:hypergeometric-weightranklisse} is our starting point and can be found in \cite[Theorem 8.4.2]{Ka90}. In what follows, we need to consider a pullback of the hypergeometric sheaf $\cH(\boldsymbol\chi,\boldsymbol\eta)$, for which we need to determine its 
geometric monodromy group $G_{\text{geom}}$. To do so, we introduce the following definitions of exceptional tuples of characters (see \cite[Corollary 8.9.2, 8.10.1]{Ka90} or \cite[Definition 3.4]{FKM15a}).

\begin{definition}
Let $\boldsymbol\chi,\boldsymbol\eta$ be an $m$-tuple and an $n$-tuple of characters of $\F_q^\times$.
\begin{itemize}
\item For $d\geqslant 1$, the pair $(\boldsymbol\chi,\boldsymbol\eta)$ is $d$-Kummer-induced if $d \mid(m,n)$ and if there exist $m/d$ and $n/d-$ tuples $\boldsymbol\chi^*$ and $\boldsymbol\eta^*$ such that $\boldsymbol\chi$ consists of all characters $\chi$ such that $\chi^d$ is a component of $\boldsymbol\chi^*$, and $\boldsymbol\eta$ consists of all characters $\eta$ such that $\eta^d$ is a component of $\boldsymbol\eta^*$.

\item Assume $m=n$. For positive integers $a,b$ such that $a+b=n,$ the pair $(\boldsymbol\chi,\boldsymbol\eta)$ is $(a,b)$-Belyi-induced if there exist characters $\alpha$ and $\beta$ with $\beta \neq 1$ such that $\boldsymbol\chi$ consists of all characters $\chi$ such that either $\chi^a=\alpha$ or $\chi^b=\beta$, and if $\boldsymbol\eta$ consists of all characters $\eta$ such that $\eta^m=\alpha \beta$.

\item Assume $m=n$. For positive integers $a, b$ such that $a+b=n,$ the pair $(\boldsymbol\chi,\boldsymbol\eta)$ is $(a,b)$-inverse-Belyi-induced if and only if $(\overline{\boldsymbol\eta},\overline{\boldsymbol\chi})$ is $(a,b)$-Belyi-induced.

\item We say that $(\boldsymbol\chi,\boldsymbol\eta)$ is Kummer-induced $($resp. Belyi-induced, inverse-Belyi-induced$)$ if there exists some $d\geqslant2$ $($resp. some $a,b\geqslant1)$ such that the pair is $d$-Kummer-induced $($resp. $(a,b)$-Belyi-induced, $(a,b)$-inverse-Belyi-induced$)$.
\end{itemize}
\end{definition}

The following lemma is borrowed directly from Katz \cite[Theorem 8.11.2]{Ka90}, giving an initial description on the 
geometric monodromy group $G_{\text{geom}}$ of $\cH(\boldsymbol\chi,\boldsymbol\eta)$.

\begin{lemma}\label{lm:hypergeometric-geometricmonodromygroup}
Suppose $m=n<p$ and write
\begin{align*}
\Lambda=\prod_{1\leqslant i\leqslant n} \chi_i \overline{\eta}_i .
\end{align*}
Assume that $(\boldsymbol\chi,\boldsymbol\eta)$ is neither Kummer-induced, Belyi-induced, nor inverse-Belyi-induced. Denote by $G^0$ the connected component of the identity in the 
geometric monodromy group $G_{\text{geom}}$ of $\cH(\boldsymbol\chi,\boldsymbol\eta)$.

Then $G^0$ is either trivial, $\SL_n, \SO_n$ or $\Sp_n.$ More precisely,
\begin{itemize}
\item If $\Lambda=1$, then $G^0$ is either $\SL_n$ or $\Sp_n;$
\item If $\Lambda \neq 1$ but $\Lambda^2=1$, then $G^0$ is either trivial or $\SO_n$ or $\SL_n;$
\item If $\Lambda^2 \neq 1$, then $G^0$ is either trivial or $\SL_n.$
\end{itemize}
\end{lemma}

\subsection{Proof of Proposition \ref{prop:charactersum}}
To begin with, we write
\begin{align}
S(\chi,\eta;\rho)
&=\eta(\omega^2)\sum_{\alpha\in\F_q} \rho(\alpha+\omega)\sum_{t\in\F_q-\{0,1\}} \chi(t)\eta((\alpha^2\omega^{-2}-1)(1-t)^{-1}+1)\nonumber\\
&=\eta(\omega^2)\sum_{\beta\in\F_q}A(\beta)B(1-\beta\omega^{-2}),\label{eq:S-A,B}
\end{align}
where
\begin{align*}
A(y)&= \sum_{\substack{\alpha\in\F_q\\ \alpha^2=y}} \rho(\alpha+\omega),
\end{align*}
and
\begin{align*}
B(y)&=\sum_{t\in\F_q-\{0,1\}} \chi(t)\eta(1-y(1-t)^{-1}). 
\end{align*}

A trivial bound for $A(y)$ shows
\begin{align}\label{eq:A(y)-trivialbound}
|A(y)|\leqslant 2.
\end{align}
The heart of our treatment to $S(\chi,\eta;\rho)$ lies in the following transformation of $B(y)$ in terms of hypergeometric sums.

\begin{lemma}\label{lm:B(y)}
Let $\chi,\eta$ be non-trivial characters of $\F_q^\times.$ For each $y\in\F_q^\times,$ we have
\begin{align*}
B(y)
&=\frac{-\tau(\chi)\tau(\eta)}{\sqrt{q}}H(y,q;\boldsymbol\chi,\boldsymbol\eta)
\end{align*}
with $\boldsymbol\chi=(\chi_0,\chi_0)$ and $\boldsymbol\eta=(\chi,\eta),$ where $\chi_0$ denotes the trivial character of $\F_q^{\times}$.
\end{lemma}

\proof
Note that
\begin{align*}
B(y)&=\mathop{\sum\sum}_{\substack{s,t\in\F_q\\ st=1}} \chi(1-s)\eta(1-yt).
\end{align*}
From orthogonality of characters of $\F_q^\times$, we may write
\begin{align*}
B(y)
&=\frac{1}{q-1}\sum_{\xi\in \fG}\Big(\sum_{s\in\F_q} \xi(s)\chi(1-s)\Big)\Big(\sum_{t\in\F_q} \xi(t)\eta(1-yt)\Big)\\
&=\frac{1}{q-1}\sum_{\xi\in \fG}\xi(y^{-1}) J(\xi,\chi)J(\xi,\eta).
\end{align*}
We would like to evaluate the Jacobi sums in terms of Gauss sums. Note that $\chi,\eta$ are both non-trivial in $\fG.$ We are in a good position to apply Lemma \ref{lm:Jacobi-Gauss}, so that for the canonical additive character $\psi$ of $\F_q$,
\begin{align*}
B(y)
&=\frac{\tau(\chi)\tau(\eta)}{q^2(q-1)}\sum_{\xi\in \fG}\xi(y^{-1})\tau(\xi)^2\overline{\tau(\xi\chi)\tau(\xi\eta)}.
\end{align*}
To complete the proof of Lemma \ref{lm:B(y)}, it suffices to prove that
\begin{align*}
\sum_{\xi\in \fG}\xi(y^{-1})\tau(\xi)^2\overline{\tau(\xi\chi)\tau(\xi\eta)}&=-q^{3/2}(q-1)H(y,q;(\chi_0,\chi_0),(\chi,\eta)),
\end{align*}
which can be verified by opening the Gauss sums by definition, and applying orthogonality again.
\endproof

From \eqref{eq:S-A,B} and Lemma \ref{lm:B(y)} it follows that
\begin{align*}
S(\chi,\eta;\rho)
&=\frac{-\eta(\omega^2)\tau(\chi)\tau(\eta)}{\sqrt{q}}T(\chi,\eta;\rho),
\end{align*}
where 
\begin{align*}
T(\chi,\eta;\rho)
&=\sum_{\alpha\in\F_q} \rho(\alpha+\omega)H(1-\alpha^2\omega^{-2},q;(\chi_0,\chi_0),(\chi,\eta)).
\end{align*}

Now Proposition \ref{prop:charactersum} follows immediately from the following assertion.

\begin{lemma}\label{lm:T(chi,eta;rho)-upperbound}
\begin{align*}
|T(\chi,\eta;\rho)|\leqslant 1000\sqrt{q}.
\end{align*}
\end{lemma}

The proof of Lemma \ref{lm:T(chi,eta;rho)-upperbound} relies heavily on the ``quasi-orthogonality" of trace functions of $\ell$-adic sheaves due to Deligne \cite{De80}, as a consequence of his proof on Riemann Hypothesis for varieties over finite fields. The following version can be found for instance in \cite[Theorem 4.1]{FKM15b}, although the statement therein is only given for prime fields.

\begin{proposition}\label{prop:RH}
Suppose $\cF_1,\cF_2$ are two geometrically irreducible $\ell$-adic sheaves on $\mathbb{P}_{\F_q}^1$ which are both pointwise pure of weight zero, and $K_1,K_2$ are the associated trace functions, respectively. If $\cF_1$ is not geometrically isomorphic to $\cF_2,$ then
\begin{align*}
\Bigg|\sum_{x\in\F_q}K_1(x)\overline{K_2(x)}\Bigg|
\leqslant 3\fc(\cF_1)^2\fc(\cF_2)^2\sqrt{q},
\end{align*}
where $\fc(\cF_j)$ denotes the conductor of $\cF_j$ as defined by Definition $\ref{def:conductor}.$
\end{proposition}

\proof[Proof of Lemma $\ref{lm:T(chi,eta;rho)-upperbound}$]

Denote by $\cH$ the sheaf with trace function $t\mapsto H(t,q;(\chi_0,\chi_0),(\chi,\eta))$ and 
geometric monodromy group $G_{\text{geom}}$.
Denote by $G^0$ the connected component of the identity in $G_{\text{geom}}$.
It is not difficult to check that
$(\boldsymbol\chi,\boldsymbol\eta)$ with $\boldsymbol\chi=(\chi_0,\chi_0)$ and $\boldsymbol\eta=(\chi,\eta)$ is neither Kummer-induced, Belyi-induced, nor inverse-Belyi-induced. We now apply Lemma \ref{lm:hypergeometric-geometricmonodromygroup} with $\Lambda=\overline{\chi\eta}$
and $m=n=2$, and it follows that $G^0$ is either trivial or $\SL_2.$
Generally speaking, it is very intricate to give a criterion for $G^0$ to be trivial.
Analysis by Katz \cite[\S 8.14--8.17]{Ka90} can, however, usually provide sufficient evidences to exclude the case of trivial $G^0$. Suppose that in our situation $G^0$ is trivial, then $\cH$ has finite geometric monodromy group $G_{\text{geom}}$. However, according to \cite[(8.17.3)]{Ka90}, this cannot happen since our first two characters are the same, equal to the trivial character $\chi_0.$ After eliminating the possibility that $G^0$ is trivial, we find $G^0$ and $G_{\text{geom}}$ must be $\SL_2.$ Note that $\SL_2$ has no finite index algebraic subgroup, we then find the geometric monodromy group of the pullback sheaf $f^*\cH$ is also $\SL_2$, where $f:\F_q\rightarrow\F_q$ is defined by $f(\alpha)=1-\alpha^2\omega^{-2}.$ Moreover, $f^*\cH$ should be geometrically irreducible of rank two.
Note that $0,1,\infty$ present all singularities of $\cH$ in $\mathbb{P}^1(\overline{\F}_q)$, which produce
four singularities of $f^*\cH$ at $0,\pm\omega$ and $\infty$. Since $\cH$ is tame everywhere (see \cite[Theorem 8.4.2]{Ka90}),
the conductor of $f^*\cH$ is $\fc(f^*\cH)=2+4+0=6$ according to Definition \ref{def:conductor}.

We now consider the function $\alpha\mapsto \rho(\alpha+\omega)$ following an argument of Katz \cite{Ka89}. For convenience, we write $F=\F_q$ and $E=\F_{q^2}.$ Given any finite-dimensional commutative $F$-algebra $A$, we denote by $\A$ the smooth affine scheme over $F$ given by ``$A$ as algebraic group over $F$", and denote by $\A^\times$ the open subscheme of $\A$ given by ``$A^\times$ as algebraic group over $F$" (should not be confused with ring of adeles used before). We apply such concepts to the cases $A=E$ and $A=F$. 
Because $\E^\times$ is a smooth, geometrically connected commutative group scheme over the finite field $F$, the Lang isogeny $1-\bFrob_F:\E^\times\rightarrow\E^\times$
makes $\E^\times$ into a $E^\times$-torsor over itself, the ``Lang torsor" $\sL$.
Note that $\rho$ can be viewed as a $\bar{\Q}_\ell$-valued character of $E^\times$,  by which it makes sense to push out the Lang torsor $\sL$ to obtain a lisse rank one $\bar{\Q}_\ell$-sheaf $\sL_\rho$ on $\E^\times$ which is pure of weight zero. We may also extend $\sL_\rho$ to $j_!\sL_\rho$ on $\E$ using the inclusion $j:\E^\times\rightarrow \E.$
For the function $v:\bF\rightarrow \E,$ $\alpha\mapsto\alpha+\omega$, the pullback sheaf $\cF:=v^*(j_!\sL_\rho)$ on $\bF$ is lisse of rank one and pure of weight zero on the open set $v^{-1}(\E^\times)$, and is zero outside. The sheaf $\cF$ is everywhere tamely ramified, because on $v^{-1}(\E^\times)$ it is lisse of order dividing that of $\rho$, hence coprime to $p$. Since there are two singularities of $\cF$ in $\mathbb{P}^1(\overline{\F}_q)$, the conductor of $\cF$ satisfies $\fc(\cF)=1+2+0=3$ according to Definition \ref{def:conductor}.

By comparing the ranks and irreducibilities, we find $f^*\cH$ and $\cF$ are not geometrically isomorphic, Lemma \ref{lm:T(chi,eta;rho)-upperbound} then follows from Proposition \ref{prop:RH}.
\endproof

\section{Proof of Theorem \ref{Main}}
\label{sec:proofcompleted}
	
	We now start the proof of Theorem \ref{Main} appealing to Theorem \ref{thm:Motohashi}, thanks to the local-global feature of which, it suffices to deal with a purely local question following the argument in \cite{BFW21+}.
	
	\subsection{Choice of test functions}
	We choose our test function $\Psi = \otimes_v' \Psi_v$ according to the data of $\pi = \otimes_v' \pi_v$ in the following three cases: 
	\begin{itemize}
	\item real places;
	\item unramified places;
	\item remaining non-archimedean places.
	\end{itemize}
The details will be presented one by one.	
	
(I) {\it  Real places}: At a real place $v$, we have $\pi_v \simeq \pi(\norm^{T_v}, \norm^{-T_v}) \otimes \sgn^{\varepsilon_v}$ for some $T_v \geq 0$ and $\varepsilon_v \in \{ 0,1 \}$. As a main result of \cite{BFW21+}, our test function $\Psi_v$ can be chosen so that
\begin{align*}
\Mt_{3,v}(\Psi_v \mid \sgn^{\varepsilon}, i\tau) =\id_{\varepsilon = \varepsilon_v} \cdot \sqrt{\pi} \frac{\cosh(\pi \tau)}{2 \Delta_v}\Big(\sum_\pm\exp \Big( - \frac{(\tau\pm T_v)^2}{2\Delta_v^2} \pm \frac{\pi}{2} \tau\Big) \Big)^2,
\end{align*}
where $\Delta_v = (1+\norm[T_v])^{\epsilon}$ with $\epsilon > 0$. Moreover, $\Mt_3(\Psi_v \mid \sigma) = 0$ if $\sigma$ is not spherical. Writing 
$$\chi_v(t) \norm[t]^s = \norm[t]^{\frac{1}{2}+ix} \sgn^{\varepsilon'}(t)$$
with $\varepsilon' \in \{ 0,1 \}$, the corresponding dual weight $\Mt_4(\Psi_v \mid \chi_v,s)$ is expressed in terms of an explicit integral transform in terms of some hypergeometric functions. In fact, tight bounds for the dual weights would suffice, and the exact formula are not necessary.
According to \cite[Theorem 1.9]{BFW21+}, for $\Re s = \frac{1}{2}$, we have
\begin{align*}
\Mt_4(\Psi_v \mid \chi_v,s) = \Mt_4(\Psi_v \mid \sgn^{\varepsilon'}, \tfrac{1}{2}+ix) \ll_A (1+\norm[x])^{-A} 
\end{align*}
for $|x|\geq (1+|T_v|) \log^2 (1+|T_v|)$ with any $A>0$, and 
\begin{align*}
\Mt_4(\Psi_v \mid \chi_v,s)\ll 1
\end{align*}
uniformly in $x\in\R.$

(II) {\it Unramified places}: At an unramifield place $\vp$ of $\pi$, we choose $\Psi_{\vp} = \id_{\Mat_2(\vo_{\vp})}$, so that it produces the relevant local zeta functions on both sides. Here we have slightly abused the terminology of ``unramified place'' to include those $\vp$ at which $\cond(\pi_{\vp})=0$ but $\vp \mid \Dif_{\bF}$. Such a local component $\Mt_{3,\vp}(\cdots)$ or $\Mt_{4,\vp}(\cdots)$ is not equal to $1$, but some power of $\Dis_{\vp} := \Nr(\Dif_{\vp})$, where $\Dif_{\vp}$ is the $\vp$-component of $\Dif_{\bF}$. The discrepancy enters into the implicit dependence on $\bF$ in Theorem \ref{Main}.

(III) {\it Remaining non-archimedean places}:
At the remaining non-archimedean places $\vp$, the test function $\Psi_{\vp}$ has been constructed explicitly in \S \ref{sec:Case1}-\ref{sec:Case3}, according to the conductor exponent $\fa(\pi)$ of $\pi$ classified in the beginning of Section \ref{sec:localcomputations}.
In each case, the weight function $\Mt_3(\Psi_{\vp} \mid \cdot)$ is non-negative on the spectrum of $\PGL_2(\bF_{\vp})$, and non-vanishing at $\pi_{\vp}$.

The following proposition summarizes Lemma \ref{LocMainBd1}, \ref{LocMainBd2-1}, \ref{LocMainBd2-2}, Corollary \ref{LocMainBd3} and \cite[Lemma 4.1 \& Corollary 4.8]{BFW21+}:

\begin{proposition}\label{prop:M4-M3-locally}
Let $\fa(\pi)=n$ with $n\in\{1,2,3,4\}.$ 
\begin{itemize}
\item For $n=1,$ $\Mt_4(\Psi_{\vp} \mid \chi_{\vp}, \frac{1}{2})$ vanishes unless $\fa(\chi_{\vp})=0,$ in which case we have
\begin{align*}
\Mt_4(\Psi_{\vp} \mid \chi_{\vp}, \tfrac{1}{2})
&\ll
\Mt_3(\Psi_{\vp} \mid \pi_{\vp})\cdot \Nr(\vp)^{\frac{1}{2}}.
\end{align*}

\item For $n=2,$ $\Mt_4(\Psi_{\vp} \mid \chi_{\vp}, \frac{1}{2})$ vanishes unless $\fa(\chi_{\vp})\leqslant1,$ in which case we have
\begin{align*}
\Mt_4(\Psi_{\vp} \mid \chi_{\vp}, \tfrac{1}{2})
&\ll
\Mt_3(\Psi_{\vp} \mid \pi_{\vp}).
\end{align*}

\item For $n=3,$ $\Mt_4(\Psi_{\vp} \mid \chi_{\vp}, \frac{1}{2})$ vanishes unless $\fa(\chi_{\vp})\leqslant1,$ in which case we have
\begin{align*}
\Mt_4(\Psi_{\vp} \mid \chi_{\vp}, \tfrac{1}{2})
&\ll
\Mt_3(\Psi_{\vp} \mid \pi_{\vp})\cdot \Nr(\vp)^{\frac{1}{2}}.
\end{align*}

\item For $n=4,$ $\Mt_4(\Psi_{\vp} \mid \chi_{\vp}, \frac{1}{2})$ vanishes unless $\fa(\chi_{\vp})\leqslant 2,$ in which case we have
\begin{align*}
\Mt_4(\Psi_{\vp} \mid \chi_{\vp}, \tfrac{1}{2})
&\ll
\Mt_3(\Psi_{\vp} \mid \pi_{\vp}).
\end{align*}

\end{itemize}
\end{proposition}

\subsection{Bounding the cubic and fourth moments}
With the above choices and estimations, we deduce, upon an obvious re-normalization of $\Psi$ and an application of \cite[Theorem 1.11]{BFW21+} (large sieve inequality over number fields)
\begin{align}\label{eq:MainTermBd-L-M3}
L(\tfrac{1}{2},\pi)^3 \ll_{\epsilon} \Mt_3(\Psi)\Cond(\pi)^{\epsilon}
\end{align}
and
\begin{align}\label{eq:MainTermBd-M4-conductor}
\Mt_4(\Psi) \ll_{\epsilon} \Cond(\pi)^{\frac{1}{2}+\epsilon}.
\end{align}
In view of Motohashi's formula (Theorem \ref{thm:Motohashi}), it remains to bound the degenerate terms. Now Theorem \ref{MainBis} (hence Theorem \ref{Main}) follows from \eqref{eq:MainTermBd-L-M3}, \eqref{eq:MainTermBd-M4-conductor} and the claims
\begin{align}\label{eq:DegenerateTermBd-D3}
\Dt_3(\Psi) \ll_{\epsilon} \Cond(\pi)^{\epsilon}
\end{align}
and
\begin{align}\label{eq:DegenerateTermBd-D4}
\Dt_4(\Psi) \ll_{\epsilon} \Cond_1(\pi)^{1/2}\Cond_3(\pi)^{1/6} \Cond(\pi)^{1/2+\epsilon}.
\end{align}

\subsection{Estimates for degenerate terms}
We first consider $\Dt_3(\Psi)$. Note that if $\pi_{\vp}$ is supercuspidal for some finite place $\vp$, then $\Mt_3(\Psi \mid \id,s)$ vanishes for any $s$, since its local component at such $\vp$ vanishes identically, from which it follows that $\Dt_3(\Psi)=0$. We now assume $\pi_{\vp}$ is not supercuspidal at every finite place $\vp$ (thus $\fa(\pi_{\vp})\leqslant2$) and suppose 
the local test function $\Psi_{\vp}$ constructed in the next section satisfy our requirements.  If the local conductor exponent $\fa(\pi_{\vp})=2$ for some finite place $\vp$, then by Lemma \ref{LocMainBd2-1} (1) we have $\Mt_3(\Psi \mid \id, s) = 0$ since $\fa(\pi(\norm_{\vp}^s, \norm_{\vp}^{-s}) \otimes \xi) = 2 > 1$ for any quadratic character $\xi$ of $\bF_{\vp}^{\times}$ with conductor exponent $1$. Thus $\Dt_3(\Psi)$ vanishes unless $\fa(\pi_{\vp}) \leq 1$  at every finite place $\vp$, in which case it follows from Lemma \ref{LocAuxBd1}
that
\begin{align*}
\Mt_3(\Psi \mid \id, s)
&= \prod_{\substack{v \mid \infty\\\varepsilon_v = 1}}\sqrt{\pi} \frac{\cos(\pi s)}{2 \Delta_v} \Bigg\{\sum_\pm\exp \Big( - \frac{(is\pm T_v)^2}{2\Delta_v^2} \mp \frac{\pi}{2} is \Big) \Bigg\}^2 \\
&\quad \cdot \prod_{\substack{\vp\\ \fa(\pi_{\vp})=1}}\frac{2\zeta_{\vp}(1+2s) \zeta_{\vp}(1-2s)}{\zeta_{\vp}( \tfrac{1}{2}+s)^2 \zeta_{\vp}( \tfrac{1}{2}-s)^2} \cdot \frac{\zeta_{\bF}(\tfrac{1}{2}+s)^3 \zeta_{\bF}( \tfrac{1}{2}-s)^3}{\zeta_{\bF}(1+2s) \zeta_{\bF}(1-2s)}.
\end{align*}
The order of vanishing of the function $s \mapsto \Mt_3(\Psi \mid \id, s)$ at $s=\tfrac{1}{2}$ is 
\begin{align*}
\ord_{s=\frac{1}{2}}\Mt_3(\Psi \mid \id, s)
&=r+|\{ \vp < \infty:\fa(\pi_{\vp})=1\}| + (-3) + 2(r-1) \\
&= 3r-5+|\{ \vp < \infty:\fa(\pi_{\vp})=1\}|,
\end{align*}
where $r := [\bF:\Q]$ is the degree of $\bF$. This order is $\leq -1$ only if $r=1$ and $|\{ \vp < \infty:\fa(\pi_{\vp})=1\}|\leq 1$. So $\Dt_3(\Psi) = 0$ unless $\bF=\Q$ and the number of ramified places of $\pi$ is at most one. In this case, we also easily deduce that 
\begin{align*} 
\Dt_3(\Psi) = \frac{1}{\zeta_{\bF}^*} \Res_{s=\frac{1}{2}} \Mt_3(\Psi \mid \id,s) \ll_{A, \epsilon}\Cond(\pi_{\infty})^{-A} \Cond(\pi_{\fin})^{\epsilon},
\end{align*}
for any $A \geqslant 1$ and $\epsilon > 0$. In particular, \eqref{eq:DegenerateTermBd-D3} holds.
	
We finally consider $\Dt_4(\Psi)$. According to (\ref{Wt4thM}) and (\ref{4thMD}) we write
$$ \Mt_4(\Psi \mid \id,s) = \zeta_{\bF}( \tfrac{1}{2}+s)^2 \zeta_{\bF}( \tfrac{1}{2}-s)^2 \prod_v \Wt_{4,v}(\Psi_v \mid \id,s). $$
It suffices to estimate the local components $\Wt_{4,v}(\Psi_v \mid \id,s)$. To this end, denote by 
$w_{n,v}(a)$ the $n$-th coefficient in the Laurent expansion of $\Wt_{4,v}(\Psi_v \mid \id,s)$ at a point $s=a$.

\begin{enumerate}
\item At a real place $v \mid \infty$, it follows from \cite[Lemma 5.2 \& Corollary 5.3]{BFW21+} that $\Wt_{4,v}(\Psi_v \mid \id,s) = \Mt_{4,v}(\Psi_v \mid \id,s)$ is regular at $s=0$, and has a double pole at $s=1$. For any integer $n \geq -2$, we have
	$$w_{n,v}(1),~w_{n,v}(0)\ll_{\epsilon,n} \Cond(\pi_v)^{\tfrac{1}{2}+\epsilon}. $$
	
\item At a finite place $\vp$ such that $\cond(\pi_{\vp})=1$, Lemma \ref{LocAuxBd1} yields
	$$ \Wt_{4,\vp}(\Psi_{\vp} \mid \id,\tfrac{1}{2}+s) = (q+1)q^{s-\frac{1}{2}}(1-q^{-\tfrac{1}{2}-s})^2, $$
	so that for each integer $n \geq 0,$
	$$ w_{n,\vp}(1)\ll_{\epsilon,n} \Cond(\pi_{\vp})^{1+\epsilon}, \quad w_{n,\vp}(0) \ll_{\epsilon,n} \Cond(\pi_{\vp})^{\epsilon}. $$
	
\item At a finite place $\vp$ such that $\cond(\pi_{\vp})=2$ and $\pi_{\vp}$ is not supercuspidal, we deduce from Lemma \ref{LocAuxBd2-1} (with a renormalization to make $\Mt_{3,\vp}(\Psi_{\vp} \mid \pi_{\vp}) \asymp 1$) that
\begin{align*} 
	\Wt_{4,\vp}(\Psi_{\vp} \mid \id,\tfrac{1}{2}+s) 
	&= \zeta_{\vp}(1)(q+1)q^{-2} ( 1-q^{-(\frac{1}{2} + s)})^2 ( 1-q^{-(\frac{1}{2} - s)})^2\\
	&\ \ \ \ \times\sum_{\pm} \Bigg\{\frac{q^2\cdot q^{-(1\pm 2s)}}{( 1-q^{-(\frac{1}{2} \pm s)})^2}
	+\frac{2 \zeta_{\vp}(1) q\cdot q^{-(1\pm 2s)}}{1-q^{-(\frac{1}{2}\pm s)}}+\xi(\pm1)\Bigg\}, 
\end{align*}
which gives
	$$  w_{n,\vp}(1),~ w_{n,\vp}(0)\ll_{\epsilon,n} \Cond(\pi_{\vp})^{\tfrac{1}{2}+\epsilon} $$
	 for each integer $n \geq 0.$	
	 
\item At a finite place $\vp$ such that $\cond(\pi_{\vp})=2$ and $\pi_{\vp}$ is supercuspidal, we apply Lemma \ref{LocAuxBd2-2} to get
\begin{align*} 
	\Wt_{4,\vp}(\Psi_{\vp} \mid \id,\tfrac{1}{2}+s) 
	&= ( 1-q^{-(\frac{1}{2} + s)})^2 ( 1-q^{-(\frac{1}{2} - s)})^2\Bigg\{ -\frac{2q^{s-\frac{1}{2}}}{1-q^{s-\frac{1}{2}}}+ (q-1) \Big( \frac{q^{s-\frac{1}{2}}}{1-q^{s-\frac{1}{2}}} \Big)^2\\
	&\ \ \ \  - \overline{\theta^{\flat}(\sqrt{\varepsilon})} (q-1) \Big( \frac{q^{-s-\frac{1}{2}}}{1-q^{-s-\frac{1}{2}}} \Big)^2 +\frac{2}{q-1} \Bigg\}, 
\end{align*}
from which we find
	$$ w_{n,\vp}(1),~ w_{n,\vp}(0)\ll_{\epsilon,n} \Cond(\pi_{\vp})^{\tfrac{1}{2}+\epsilon}. $$
	 for each integer $n \geq 0.$
	
\item At a finite place $\vp$ such that $\cond(\pi_{\vp})=3$ and $\pi_{\vp}$ is supercuspidal, we deduce from Lemma \ref{LocAuxBd3} (with a renormalization to make $\Mt_{3,\vp}(\Psi_{\vp} \mid \pi_{\vp}) \asymp 1$) that
	$$ \Wt_{4,\vp}(\Psi_{\vp} \mid \id,\tfrac{1}{2}+s) =q^{s+\frac{1}{2}} (q+1) ( 1-q^{-(\frac{1}{2} + s)})^2 ( 1-q^{-(\frac{1}{2} - s)})^2\Big( -\frac{1}{q-1} + \frac{q^{s-\frac{1}{2}}}{1-q^{s-\frac{1}{2}}} \Big)^2, $$
	which gives
	$$ w_{n,\vp}(1)\ll_{\epsilon,n} \Cond(\pi_{\vp})^{\tfrac{2}{3}+\epsilon}, \quad  w_{n,\vp}(0)\ll_{\epsilon,n}\Cond(\pi_{\vp})^{-\tfrac{1}{3}+\epsilon}$$
		 for each integer $n \geq 0.$
	
\item At a finite place $\vp$ such that $\cond(\pi_{\vp})=4$ and $\pi_{\vp}$ is not supercuspidal, we import the relevant result leading to \cite[(5.1)]{BFW21+}, getting
\begin{align*} 
	\Wt_{4,\vp}(\Psi_{\vp} \mid \id,\tfrac{1}{2}+s)
	&= \zeta_{\vp}(1) ( 1-q^{-(\frac{1}{2} + s)})^2 ( 1-q^{-(\frac{1}{2} - s)})^2 \\
	&\ \ \ \times\sum_{\pm} \Bigg\{\frac{q^2\cdot q^{-2(1\pm 2s)}}{( 1-q^{-(\frac{1}{2} \pm s)})^2}
	+\frac{2 \zeta_{\vp}(1) q\cdot q^{-2(1\pm 2s)}}{1-q^{-(\frac{1}{2}\pm s)}}+q^{-(1\pm 2s)}\Bigg\}, 
\end{align*}
so that for each integer $n \geq 0,$
	$$ w_{n,\vp}(1),~ w_{n,\vp}(0) \ll_{\epsilon,n} \Cond(\pi_{\vp})^{\tfrac{1}{2}+\epsilon}. $$
\end{enumerate}

The above bounds for $w_{n,v}(1),w_{n,v}(0)$ would be used to evaluate the residues 
	\begin{align}
\Res_{s=1} \Mt_4(\Psi \mid \mathbbm{1},s) ,\ \ \Res_{s=0} \Mt_4(\Psi \mid \mathbbm{1},s)
	\end{align}
in \eqref{DGF}, so that \eqref{eq:DegenerateTermBd-D4} can be deduced readily.
Now we are done!

\appendix

\section{Remarks on Period Approach to Motohashi's Formula}
\label{sec:appendix}
	
	\subsection{Recall on period approach}
	
	Different methods have been exploited by various authors in order to understand the structural reason under Motohashi's formula, see \cite{BrM05, BHKM20} for example. In \cite[\S 4.3.3]{MV06} and \cite[\S 4.5.3]{MV10} Michel and Venkatesh sketched a period approach, which suggests explaining Motohashi's formula as a special case of the \emph{strong Gelfand configurations}, proposed by Reznikov \cite{Re08}, as follows
\begin{equation}
\begin{matrix}
	& & \GL_2 \times \GL_2 & & \\
	& \nearrow & & \nwarrow & \\
	\GL_1 \times \GL_1 & & & & \GL_2 \\
	& \nwarrow & & \nearrow & \\
	& & \GL_1 & &
\end{matrix}.
\label{MVGraph}
\end{equation}

	We illustrate the details in the case relevant to this paper. We consider the regularized integral 
\begin{equation} 
	\int_{\bF^{\times} \backslash \A^{\times}}^{\reg} \eis(s_1,f_1) \cdot \eis(s_2,f_2) (a(t)) \ud^{\times}t 
\label{RegIntPerMF}
\end{equation}
	along the diagonal torus of the product of two Eisenstein series constructed from smooth vectors $f_1,f_2 \in V_{\mathbbm{1}, 0} = \Ind_{\gp{B}(\A)}^{\GL_2(\A)} \id$. On one hand, one expects a suitable automorphic Fourier inversion formula for this product, so that the projection on $V_{\pi}$ for a cuspidal representation $\pi$ gives the contribution
\begin{equation} 
	\sum_{\varphi \in \Bas(\pi)} \Pairing{\eis(s_1,f_1) \cdot \eis(s_2,f_2)}{\varphi} \int_{\bF^{\times} \backslash \A^{\times}} \varphi(a(t)) \ud^{\times}t. 
\label{PerProj}
\end{equation}
	By the Rankin--Selberg theory for $\GL_2 \times \GL_1$, the above integral represents $L(1/2, \pi)$. By the Rankin--Selberg theory for $\GL_2 \times \GL_2$, the above inner product $L(1/2+s_1,\pi) L(1/2+s_2, \pi)$. Hence (\ref{PerProj}) represents a certain cubic moment of $\GL_2$ $L$-functions. On the other hand, one expects a Parseval-type identity over $\bF^{\times} \backslash \A^{\times}$, which expresses (\ref{RegIntPerMF}) as
\begin{multline}
\sum_{\chi \in \widehat{\R_+ \bF^{\times} \backslash \A^{\times}}}\int_{-\infty}^{\infty}\Big( \int_{\bF^{\times} \backslash \A^{\times}}^{\reg} \eis(s_1,f_1)(a(t)) \chi(t) \norm[t]_{\A}^{i\tau} \ud^{\times}t \Big) \cdot \Big( \int_{\bF^{\times} \backslash \A^{\times}}^{\reg} \eis(s_2,f_2)(a(t)) \chi^{-1}(t) \norm[t]_{\A}^{-i\tau} \ud^{\times}t \Big) \frac{\ud\tau}{2\pi}.
\label{PerParseval}
\end{multline}
	Again by the Rankin--Selberg theory for $\GL_2 \times \GL_1$, the two inner integrals represent $L(i\tau+s_1, \chi) L(i\tau-s_1, \chi)$ and $L(-i\tau+s_2, \chi^{-1}) L(-i\tau-s_2, \chi^{-1})$, respectively. Hence (\ref{PerParseval}) represents a certain fourth moment of $\GL_1$ $L$-functions.
	
	Michel and Venkatesh noticed the non-trivial convergence issues in the above sketch, but did not provide any hint of solution. Recently, Nelson \cite{Ne20} announced a solution to these issues by the theory of regularized integrals, which is favourable to the application in his paper. We have initiated a comparison between Nelson's period method and the first author's distributional method in \cite[Appendix]{Wu22}, and found some possible disagreement of the two versions mainly on the cubic moment side. We shall refine the comparison here, and clarify the non-trivial gap between Nelson's version and ours.

	\subsection{Regularized integrals v.s. meromorphic continuations}
	The theory of regularized integrals has been playing important roles when establishing meromorphic continuations, for instance in Tate's thesis and theta correspondences. This fruitful theory turns out to be very powerful, but we believe that it can only cover the scope of methods of meromorphic continuations for a very special class of functions, say Mellin transforms as shown below. We now explain this viewpoint over $\R_{>0}$, in the framework of Nelson's version \cite[\S 5]{Ne20}.

	Recall that a finite function $\phi$ on $\R_{>0}$ is so defined that the translates $\varphi_a(t) := \phi(at)$ for $a \in \R_{>0}$ span a finite dimensional space. Concretely, a finite function is a linear combination of functions of the form
	$$ t^{\alpha} (\log t)^n $$
for $\alpha \in \R$ and $n \in \Z_{\geq 0}$.

	A (strongly) regularizable function $f: \R_{>0} \to \C$ is so defined that there exist two finite functions $\phi_0$ and $\phi_{\infty}$ satisfying
\begin{itemize}	
	\item $f(t) - \phi_0(t) = O(t^A)$ for any $A > 1$ as $t \to 0^+$;
	\item $f(t) - \phi_{\infty}(t) = O(t^{-A})$ for any $A > 1$ as $t \to +\infty$.
\end{itemize}
	For such a function, the Mellin transform
	$$ I(s) = \int_0^{\infty} (f(t) - \phi_\infty(t)) t^s \frac{\ud t}{t}, \quad \Re s \gg 1 $$
	admits a meromorphic continuation to $s \in \C$. If $I(s)$ is regular at $s=0$\footnote{The regularity at $0$ can be removed to give an extension useful for certain applications. See \cite[\S 2]{Wu19_TF}.},
	the regularized integral of $f$ is defined to be
	$$ \int_{\R_{>0}}^{\reg} f(t) \frac{\ud t}{t} := I(0). $$

	Some important examples of regularizable functions are constructed from regularizable functions on $\GL_2(\bF) \backslash \GL_2(\A)$, whose definition is cumbersome to recall. But they are modelled by products of Eisenstein series. For example, the function $\varphi(s_1,s_2) := \eis(s_1,f_1) \cdot \eis(s_2,f_2)$ in the integrand of (\ref{RegIntPerMF}) is regularizable on $\GL_2(\bF) \backslash \GL_2(\A)$, with the \emph{essential constant term} given by
	$$ \varphi(s_1,s_2)_{\gp{N}}^* = \left( f_1(s_1) + \Intw f_1(s_1) \right) \cdot \left( f_2(s_2) + \Intw f_2(s_2) \right), $$
	where $\Intw$ is the intertwining operator on (the induced model of) the principal series representations. The corresponding function $f$ on $\R_{>0}$ defined by
\begin{equation} \label{eq:AuxRegFR+}
	f(t) := \int_{\bF^{\times} \backslash \A^{(1)}} \varphi(s_1,s_2)(a(ty)) \ud^{\times} y 
\end{equation}
	is regularizable. The associated finite functions are
	$$ \phi_{\infty}(t) := \int_{\bF^{\times} \backslash \A^{(1)}} \varphi(s_1,s_2)_{\gp{N}}^*(a(ty)) \ud^{\times} y, \quad \phi_0(t) := \int_{\bF^{\times} \backslash \A^{(1)}} \varphi(s_1,s_2)_{\gp{N}}^*(a(ty)w) \ud^{\times} y. red{\text{what is $w$}}$$
	
	Consider the classical Bessel function $J_{\nu}(t)$, say with $\nu \in i\R$, which has the following formula \cite[10.22.43]{OLBC10} (Weber's formula, after Schafheitlin)
	$$ \int_0^{\infty} t^s J_{\nu}(t) \frac{\ud t}{t} = \frac{1}{\sqrt{2}} \frac{\Gamma( \frac{s+\nu}{2})}{\Gamma( \frac{2-s+\nu}{2} )}, \quad 0 < \Re(s) < \frac{3}{2}. $$
	Note that the above integral is conditionally convergent in the sense of Cauchy, and is absolutely convergent if $0 < \Re(s) <\frac{1}{2}$. Consider $f(t) := t J_{\nu}(t)$. It is not a regularizable function\footnote{It is not regularizable even taking the extension by the first author into account.}, because its asymptotic main term at the infinity
	$$ J_{\nu}(t) = \sqrt{\frac{2}{\pi t}} \cos \left( t - \frac{\nu \pi}{2} - \frac{\pi}{4} \right) + O(t^{-1}) $$
	is not a finite function on $\R_{>0}$. Nevertheless, the Mellin transform
	$$ I(s) := \int_0^{\infty} f(t) t^s \frac{\ud t}{t}, \quad -1 < \Re s < - \tfrac{1}{2} $$
	does admit a meromorphic continuation to $s \in \C$ with $I(0) = 1/\sqrt{2}$.

	\subsection{Remarks on period approach}

	The regularizable function $\varphi(s_1,s_2) := \eis(s_1,f_1) \cdot \eis(s_2,f_2)$ is never integrable along $\bF^{\times} \backslash \A^{\times}$ for any $s_1,s_2$. The regularized integral (\ref{RegIntPerMF}) is by definition the analytically continued value at $s_0=0$ of
\begin{equation}
	I(s_0,s_1,s_2) := \int_{\bF^{\times} \backslash \A^{\times}} \left( \varphi(s_1,s_2) - \varphi(s_1,s_2)_{\gp{N}}^* \right)(a(t)) \norm[t]_{\A}^{s_0} \ud^{\times}t, \quad \Re s_0 \gg 1,
\label{RegIntPerMFDef}
\end{equation}
	The function $\varphi(s_1,s_2)$ is never square integrable, either. One brings the square integrability by subtracting it by some linear combination of Eisenstein series $\Reis(s_1,s_2)$, and regroups the integrand of (\ref{RegIntPerMFDef}) as (omitting the parameters $s_1,s_2$ for simplicity)
\begin{equation}
	\varphi - \varphi_{\gp{N}}^* = \left\{ (\varphi-\Reis) - (\varphi-\Reis)_{\gp{N}} \right\} + \left\{ \Reis - \Reis_{\gp{N}} \right\} + \left\{ \varphi_{\gp{N}} - \varphi_{\gp{N}}^* \right\}.
\label{ReGp}
\end{equation}
	The precise construction of $\Reis(s_1,s_2)$ depends on the region of the parameters $(s_1,s_2)$. For example, Nelson works on the region 
\begin{equation} \label{eq:AuxParReg1}
	D_1 = \left\{ (s_1,s_2) \in \C^2 \ \middle| \ \norm[s_1], \norm[s_2] < \epsilon \right\}
\end{equation}	
	 near the origin point, for which $\Reis(s_1,s_2)$ is the sum of $4$ Eisenstein series constructed from all the $4$ holomorphic sections contained in $\varphi_{\gp{N}}^*$:
	 $$ f_1(s_1) \cdot f_2(s_2), \quad f_1(s_1) \cdot \Intw f_2(s_2), \quad \Intw f_1(s_1) \cdot f_2(s_2), \quad \Intw f_1(s_1) \cdot \Intw f_2(s_2). $$
	 The first author works on another region
\begin{equation} \label{eq:AuxParReg2}
	D_2 = \left\{ (s_1,s_2) \in \C^2 \ \middle| \ \Re s_1, \Re s_2 - \Re s_1 > 1/2 \right\},
\end{equation}	
	for which $\Reis(s_1,s_2)$ is the sum of $4$ Eisenstein series constructed from
	$$ f_1(s_1) \cdot f_2(s_2), \quad \Intw f_1(s_1) \cdot f_2(s_2). $$

	Now that $\varphi - \Reis$ is square integrable by construction, the contribution of the first term in (\ref{ReGp}) to (\ref{RegIntPerMFDef}) gives the main distribution
\begin{equation} \label{eq:MPerM}
	I^M(s_0,s_1,s_2) := \int_{\bF^{\times} \backslash \A^{\times}} \left\{ (\varphi-\Reis) - (\varphi-\Reis)_{\gp{N}} \right\}(a(y)) \norm[y]_{\A}^{s_0} \ud^{\times}y 
\end{equation}	
	which represents a certain cubic moment by applying a spectral decomposition of $\varphi - \Reis$. It contains (\ref{PerProj}) in part. The verification that the resulted expressions of $I(s_0,s_1,s_2)$ in different regions of $(s_1,s_2)$ are meromorphic continuations of each other would require some equation similar to Tate's fundamental equation (\ref{TateIntG}). Note that in (\ref{TateIntG}) the first line (resp. last line) with one term is valid/absolutely convergent for $\Re s > 1$ (resp. $\Re s < 0$), while the two lines in the middle with four terms holds for all $s \in \C$. This should be a theoretic explanation of the discrepancy on the number of degenerate terms noticed in \cite[Footnote 2]{Wu22}, although the details might be non-trivial in practice. Already, one gets the feeling that $I(s_0,s_1,s_2)$ may look simpler in some region of parameters than in others.
	
	The \emph{real challenge in the period approach} is the meromorphic continuation of the contribution to (\ref{RegIntPerMFDef}) of the third term in (\ref{ReGp}), namely
\begin{equation} \label{eq:ChDegPerM}
	I^{\sharp}(s_0,s_1,s_2) := \int_{\bF^{\times} \backslash \A^{\times}} \left( \varphi(s_1,s_2)_{\gp{N}} - \varphi(s_1,s_2)_{\gp{N}}^* \right)(a(t)) \norm[t]_{\A}^{s_0} \ud^{\times}t.
\end{equation}
\begin{remark}
	Note that $\varphi(s_1,s_2)_{\gp{N}}$, not $\varphi(s_1,s_2)$, appears above. Hence its meromorphic continuation goes beyond any naive application of the theory of regularized integrals. In fact, the corresponding function on $\R_{>0}$, constructed similarly to (\ref{eq:AuxRegFR+}), seems to have an oscillatory asymptotic behavior at $0$ similar to the one of $J_{\nu}(t)$ at $\infty$. Note also that it is independent of $\Reis$, hence can not be avoided by changing the region of parameters $(s_1,s_2)$. 
\end{remark}	
	
	Nelson claims a partial solution by imposing some conditions, call them \textrm{NVC} (Nelson's vanishing conditions), on the test functions $f_1,f_2$, so that $I^{\sharp}(s_0,s_1,s_2)$ vanishes identically for all $(s_1,s_2)$. The first author provides in \cite[\S 7 Appendix]{Wu22} a complete solution by relating (\ref{RegIntPerMFDef}) with the distributional version via the construction of Eisenstein series via Godement sections, so that every degenerate term can be identified with some residue of the component $I^M(s_0,s_1,s_2 \mid \id,s)$ of the main term corresponding to the continuous spectra $\pi(\id,s) = \pi(\norm_{\A}^s, \norm_{\A}^{-s})$. Since $I^M(s_0,s_1,s_2 \mid \id,s)$ has meromorphic continuation to $(s_0,s_1,s_2,s) \in \C^4$ given explicitly in terms of the Godement--Jacquet and the Rankin--Selberg zeta integrals, the desired meromorphic continuation of each degenerate term follows. A simple computation shows that $I^{\sharp}(\cdots)$ corresponds to
	$$ DS_4(s_0,s_1,s_2) := \int_{(\A^{\times})^3} \OFour_2 \OFour_4 \Psi \left( \begin{pmatrix} -t_1^{-1}t_2t & t_1t^{-1} \\ t_2t & t^{-1} \end{pmatrix} \right) \norm[t]_{\A}^{s_0} \norm[t_1]_{\A}^{s_1} \norm[t_2]_{\A}^{s_2} \ud^{\times}t \ud^{\times}t_1 \ud^{\times}t_2 $$
	in the distributional version (see \cite[\S 1.5.4, (4.7) \& Corollary 4.10 (3)]{Wu22}). Its vanishing looks exotic, and does not seem to be satisfied by the test functions used in \cite{BFW21+}.
	
\begin{remark}
	There is another degenerate term in the first author's version, namely $DS_0$ given in \cite[\S 1.5.4, (4.5) \& Corollary 4.10 (1)]{Wu22}, which is supported in the complement of $\GL_2(\A)$ in $\Mat_2(\A)$. It does vanish identically for reasonable test functions such as those used in \cite{BFW21+}.
\end{remark}

	As a conclusion, it seems difficult to regard the theory of regularized integrals, say in its current form, as an adequate tool for a complete version of Motohashi's formula without NVC. On the other hand, the relation between degenerate terms and main terms is independent of the version of the formula. For the good of the development of the period method, it is an interesting question to study the meromorphic continuation of $I^{\sharp}(s_0,s_1,s_2)$ without appealing to Godement sections, so that one may understand the residues of $s \mapsto I^M(s_0,s_1,s_2 \mid \id,s)$ from a different perspective, hopefully more convenient for generalizations.

\bibliographystyle{acm}
	
\bibliography{PGL2bib}
	
\end{document}